\documentclass[11pt,reqno]{amsart}

\usepackage[T1]{fontenc}
\usepackage{lmodern}
\usepackage{microtype}
\usepackage{amsmath,amssymb,amsthm,mathtools}
\usepackage{aliascnt}
\usepackage{esint}
\usepackage{enumitem}
\usepackage{graphicx}
\usepackage{xcolor}
\usepackage[colorlinks=true,
  linkcolor=blue!55!black,
  citecolor=blue!55!black,
  urlcolor=blue!55!black,
  pdftitle={Sharp Hausdorff Bounds for the Interior Singular Set of Convex k-Hessian Solutions},
  pdfauthor={Xiyu Hu},
  pdfsubject={Fully nonlinear elliptic PDE, convexity, and geometric measure theory},
  pdfkeywords={k-Hessian equation, singular set, Hausdorff measure, partial regularity, Sobolev regularity, W2,1}]{hyperref}
\usepackage[nameinlink,capitalize,noabbrev]{cleveref}
\usepackage[margin=1.10in]{geometry}

\allowdisplaybreaks
\numberwithin{equation}{section}

\newtheorem{theorem}{Theorem}[section]

\newaliascnt{proposition}{theorem}
\newtheorem{proposition}[proposition]{Proposition}
\aliascntresetthe{proposition}

\newaliascnt{lemma}{theorem}
\newtheorem{lemma}[lemma]{Lemma}
\aliascntresetthe{lemma}

\newaliascnt{corollary}{theorem}
\newtheorem{corollary}[corollary]{Corollary}
\aliascntresetthe{corollary}

\newaliascnt{claim}{theorem}

\aliascntresetthe{claim}

\theoremstyle{definition}
\newaliascnt{definition}{theorem}

\aliascntresetthe{definition}

\newaliascnt{remark}{theorem}
\newtheorem{remark}[remark]{Remark}
\aliascntresetthe{remark}

\crefname{equation}{equation}{equations}
\crefname{theorem}{theorem}{theorems}
\crefname{proposition}{proposition}{propositions}
\crefname{lemma}{lemma}{lemmas}
\crefname{corollary}{corollary}{corollaries}
\crefname{remark}{remark}{remarks}
\crefname{definition}{definition}{definitions}
\crefname{claim}{claim}{claims}
\crefname{section}{section}{sections}

\newcommand{\R}{\mathbb{R}}
\newcommand{\Sym}{\operatorname{Sym}}
\newcommand{\tr}{\operatorname{tr}}
\newcommand{\aff}{\operatorname{aff}}
\newcommand{\rank}{\operatorname{rank}}
\newcommand{\dist}{\operatorname{dist}}

\newcommand{\Gr}{\operatorname{Gr}}
\newcommand{\cH}{\mathcal{H}}
\newcommand{\cP}{\mathcal{P}}
\newcommand{\cM}{\mathcal{M}}
\newcommand{\cC}{\mathcal{C}}
\newcommand{\Reg}{\operatorname{Reg}}
\newcommand{\Sing}{\operatorname{Sing}}
\newcommand{\fintavg}{\fint}
\newcommand{\eps}{\varepsilon}
\newcommand{\ol}{\overline}
\newcommand{\ip}[2]{\langle #1,#2\rangle}

\title[Sharp Hausdorff bounds for interior singular sets]
{Sharp Hausdorff Bounds for the Interior Singular Set of Convex $k$-Hessian Solutions}

\author{Xiyu Hu}
\date{}


\begin{document}

\begin{abstract}
Let $2\le k\le n$, let $\Omega\subset\R^n$ be open and convex, and let $u$ be a convex viscosity solution of
\[
\sigma_k(D^2u)=1\qquad\text{in }\Omega.
\]
We prove that the set on which $u$ fails to be locally $C^2$ has vanishing $(n-1)$-dimensional Hausdorff measure.  In the intermediate range $3\le k<n$, this gives a codimension-one refinement of the known almost-everywhere partial regularity, and the exponent is sharp.  More generally, for a convex viscosity subsolution of $\sigma_k(D^2u)\ge\lambda>0$, we obtain Hausdorff bounds for strata defined by the affine dimension of all supporting contact sets.  The proof combines a support-dependent Chou--Wang barrier argument, an estimate for the product of the smallest $k$ semiaxes of a John ellipsoid, and Mooney's convex section-covering theorem.  As a direct analytical consequence, the full distributional Hessian is absolutely continuous and $u\in W^{2,1}_{\mathrm{loc}}(\Omega)$, yielding a $k$-Hessian counterpart of the $W^{2,1}$ regularity known for singular Monge--Amp\`ere solutions.  In a logically separate structural part, we characterize the distinguished number of flat directions, $n-k+1$, by an asymptotic infimum mean-value formula over affine sections, and explain how this mean-value heuristic leads to the supporting-contact geometry used in the proof.
\end{abstract}

\maketitle

\section{Introduction}

For $A\in\Sym(n)$, let
\[
\sigma_k(A)=\sum_{1\le i_1<\cdots<i_k\le n}
\lambda_{i_1}(A)\cdots\lambda_{i_k}(A)
\]
denote the $k$th elementary symmetric polynomial of the eigenvalues of $A$.  The associated G\aa rding cone is
\[
\Gamma_k:=\{A\in\Sym(n):\sigma_j(A)>0\text{ for }1\le j\le k\}.
\]
We use the standard admissible viscosity interpretation of the equation, recalled in \Cref{sec:preliminaries}.

The regularity theory of the $k$-Hessian equation interpolates between the Laplace equation, $k=1$, and the Monge--Amp\`ere equation, $k=n$.  The classical Dirichlet theory, nonclassical solution theory, and weak Hessian-measure framework were developed in foundational work including \cite{CNS1985,Urbas1990,Trudinger1995,Trudinger1997,TrudingerWang1999}.  Singular convex solutions occur for $k\ge3$, while positive right-hand side forces substantial geometric nondegeneracy.  For the Monge--Amp\`ere equation, Mooney proved that the non-strictly-convex set has vanishing $\cH^{n-1}$ measure and constructed examples showing that the exponent is optimal \cite{Mooney2015}.  The convexity and rank structure of Hessian equations has also been developed through full-rank and constant-rank results, including \cite{GuanMa2003,CaffarelliGuanMa2007}; a strong maximum principle for generalized Monge--Amp\`ere-type solutions at strictly convex points was proved in \cite{JianTu2025}.  For the quadratic Hessian equation, convex viscosity solutions of the constant-right-hand-side equation are smooth \cite{Mooney2021,MooneyRemarks2025}, and a recent result gives $C^2$ regularity for positive Lipschitz right-hand side \cite{ChenJianTuZhou2026}.  In contrast, for general $k$, Fan proved that a $k$-convex viscosity solution with positive Lipschitz right-hand side is $C^{2,\alpha}$ outside a closed Lebesgue-null set, using a nonhomogeneous extension of Savin's small-perturbation theorem \cite{Fan2025,Savin2007}.

The purpose of this paper is to obtain the sharp codimension-one Hausdorff estimate for convex solutions throughout the nonlinear range $2\le k\le n$.  The new range is $3\le k<n$.

\begin{theorem}[Codimension-one partial regularity]\label{thm:main}
Let $2\le k\le n$, let $\Omega\subset\R^n$ be open and convex, and let $u\in C(\Omega)$ be a convex viscosity solution of
\begin{equation}\label{eq:main-equation}
\sigma_k(D^2u)=1
\end{equation}
in the viscosity sense.  Define
\[
\Reg(u):=\{x\in\Omega:u\text{ is }C^2\text{ in a neighborhood of }x\},
\qquad
\Sing(u):=\Omega\setminus\Reg(u).
\]
Then
\begin{equation}\label{eq:main-Hausdorff}
\cH^{n-1}_{\mathrm{loc}}\bigl(\Sing(u)\bigr)=0.
\end{equation}
At every point of $\Reg(u)$, the solution is in fact smooth.
\end{theorem}

For $k=n$, \Cref{thm:main} recovers the Hausdorff conclusion in \cite{Mooney2015}; for $k=2$, it is weaker than full regularity.  To the best of the author's knowledge, the conclusion is new for $3\le k<n$.  The exponent cannot be replaced by $n-1-\delta(n,k)$: cylindrical extensions of Mooney's singular Monge--Amp\`ere examples have singular sets of Hausdorff dimension arbitrarily close to $n-1$; see \Cref{sec:sharpness}.

A significant analytical consequence of \Cref{thm:main} is Sobolev regularity of the full distributional Hessian.  For strictly convex Alexandrov solutions of the Monge--Amp\`ere equation with density bounded above and below, De Philippis and Figalli proved that $D^2u\in L\log^m L_{\mathrm{loc}}$ for every finite $m$, and hence $u\in W^{2,1}_{\mathrm{loc}}$ \cite{DePhilippisFigalli2013}.  De Philippis, Figalli, and Savin, and independently Schmidt, subsequently obtained $W^{2,1+\varepsilon}_{\mathrm{loc}}$ estimates \cite{DePhilippisFigalliSavin2013,Schmidt2013}.  These advances build on Caffarelli's localization and section-normalization theory \cite{Caffarelli1990Localization,Caffarelli1990W2p} and exploit the full covariance of the Monge--Amp\`ere equation under volume-preserving affine transformations.  Mooney developed a complementary geometric route for singular Monge--Amp\`ere solutions, obtaining $W^{2,1}$ estimates from the geometry and size of the non-strictly-convex set \cite{MooneyW21Estimate2015,Mooney2015}.

For $2\le k<n$, the $k$-Hessian operator is not invariant under general volume-preserving affine changes of variables, so the affine normalization underlying the Monge--Amp\`ere theory does not transfer directly.  Our approach follows and extends Mooney's geometric mechanism.  Once \Cref{thm:main} gives $\cH^{n-1}_{\mathrm{loc}}(\Sing(u))=0$, a codimension-one growth bound for the positive matrix-valued Hessian measure rules out concentration on the singular set and yields the following corollary.  This provides a $k$-Hessian counterpart of the $W^{2,1}$ conclusion without using full affine invariance.  The argument is qualitative, however: it does not presently yield $L\log L$ or $W^{2,1+\varepsilon}$ estimates.

\begin{corollary}[$W^{2,1}$ regularity]\label{cor:W21-intro}
Under the assumptions of \Cref{thm:main},
\[
u\in W^{2,1}_{\mathrm{loc}}(\Omega).
\]
\end{corollary}

\subsection{Supporting contact strata}

The proof is organized by the geometry of supporting affine functions.  We use the convex-analytic subdifferential
\[
\partial u(x):=\{p\in\R^n:u(y)\ge u(x)+p\cdot(y-x)\text{ for every }y\in\Omega\}.
\]
For $x\in\Omega$ and $p\in\partial u(x)$, write
\[
\ell_{x,p}(y):=u(x)+p\cdot(y-x)
\]
and, whenever $B_r(x)\Subset\Omega$, define the local contact set
\begin{equation}\label{eq:contact-set}
K_{x,p}(r):=\{y\in B_r(x):u(y)=\ell_{x,p}(y)\}.
\end{equation}
The sets $K_{x,p}(r)$ are convex and nested as $r\downarrow0$.  Consequently, the integer-valued function
\[
r\longmapsto \dim\aff K_{x,p}(r)
\]
is nonincreasing and eventually constant.  We define the local contact dimension by
\begin{equation}\label{eq:contact-dimension}
d_u(x,p):=\lim_{r\downarrow0}\dim\aff K_{x,p}(r).
\end{equation}
For $m\in\{0,\ldots,n\}$, let
\begin{equation}\label{eq:contact-stratum}
\cC_m(u):=
\{x\in\Omega:d_u(x,p)\ge m\text{ for every }p\in\partial u(x)\}.
\end{equation}
The quantifier ``for every supporting slope'' is essential because the convex covering theorem used below is formulated for all sections at a point.  The convexity of $\Omega$ is used here only to make this global subdifferential convenient; by \Cref{lem:local-global-subgradient} below, a slope supporting on any neighborhood of an interior point automatically supports throughout $\Omega$.

The geometric part of the argument only requires a positive viscosity lower bound.

\begin{theorem}[Contact stratification]\label{thm:stratification}
Let $2\le k\le n$, and let $u\in C(\Omega)$ be convex and a viscosity subsolution of
\begin{equation}\label{eq:positive-subsolution}
\sigma_k(D^2u)\ge\lambda>0.
\end{equation}
Set
\begin{equation}\label{eq:q-def}
q:=n-k+1.
\end{equation}
For each integer $m$ with $q\le m\le n-1$,
\begin{equation}\label{eq:strata-bound}
\cH^{2n-m-k}_{\mathrm{loc}}\bigl(\cC_m(u)\bigr)=0.
\end{equation}
Moreover, $\cC_n(u)=\varnothing$.
\end{theorem}

The PDE input connecting \Cref{thm:stratification} to \Cref{thm:main} is the following local criterion.

\begin{theorem}[Small supporting contact sets are regular]\label{thm:contact-regularity}
Let $u$ be a convex viscosity solution of \eqref{eq:main-equation}.  If there exist $x_0\in\Omega$ and $p_0\in\partial u(x_0)$ such that
\begin{equation}\label{eq:small-contact-intro}
d_u(x_0,p_0)\le n-k,
\end{equation}
then $u$ is smooth in a neighborhood of $x_0$.  Consequently,
\begin{equation}\label{eq:sing-in-Cq}
\Sing(u)\subset\cC_{n-k+1}(u).
\end{equation}
\end{theorem}

Taking $m=q$ in \Cref{thm:stratification} and using
\[
2n-q-k=n-1
\]
proves \Cref{thm:main}.

\subsection{Main mechanism}

After subtracting a supporting affine function, suppose that $u(0)=0$ and $u\ge0$.  The proof has three main ingredients.

\begin{enumerate}[label=\textup{(\roman*)},leftmargin=2.2em]
\item If the local contact set $\{u=0\}$ is contained in a subspace of dimension at most $n-k$, one can build a $k$-admissible saddle quadratic with $k$ strongly positive directions and $n-k$ negative directions.  Smooth Dirichlet approximation, the interior gradient estimate, and the Chou--Wang Pogorelov estimate then yield a local Hessian bound and smoothness.

\item Let
\[
K_h:=\{u\le h\}\cap\ol B_R
\]
and let
\[
a_1(h)\ge\cdots\ge a_n(h)>0
\]
be the semiaxes of a John ellipsoid of $K_h$.  An ellipsoidal paraboloid touches $u$ from above at an interior point.  The viscosity inequality gives
\begin{equation}\label{eq:outline-axis-product}
a_q(h)\cdots a_n(h)
\le C(n,k)\lambda^{-1/2}h^{k/2}.
\end{equation}
Thus positive $k$-Hessian density controls precisely the smallest $k$ John axes.

\item If a supporting contact set contains an $m$-dimensional simplex, then $a_m(h)$ is bounded below independently of $h$.  For the strictly convex perturbation
\[
v(x):=u(x)+\tfrac12|x|^2,
\]
the corresponding sections satisfy
\begin{equation}\label{eq:outline-section-decay}
|S^v_{h,p+x}(x)|\le C_{x,p}h^{(m+k)/2}.
\end{equation}
Mooney's convex section-covering theorem converts \eqref{eq:outline-section-decay} into \eqref{eq:strata-bound}.
\end{enumerate}

The proof of \Cref{thm:main} is therefore entirely finite-scale and convex-geometric after the support-dependent regularity criterion has been established.  It does not use the sectional mean-value formula stated below.  The qualitative loss occurs when the limiting condition ``the supporting contact set has affine dimension at least $m$'' is converted into a finite-scale simplex: its inradius may depend on the point and on the supporting slope.  Once a relative contact ball of radius $\rho$ is prescribed, however, the section estimate is quantitative; see \eqref{eq:section-decay-quantitative} and \Cref{prop:section-decay}.  The absence of a uniform lower bound for $\rho$, together with possible rotation and redistribution among the smallest John axes, is the obstruction to a scale-by-scale packing estimate; see \Cref{sec:quantitative-limitations}.

\subsection{A companion sectional viewpoint}

The search for a nonlinear mean-value principle is not used as a lemma in the proof above, but it identifies the same threshold
\[
q=n-k+1
\]
from the homogeneous equation itself.  This provides a conceptual explanation for why $q$ is the critical number of supporting flat directions in \Cref{thm:contact-regularity,thm:stratification}.

For $E\in\Gr(q,n)$, let $B_r^E(x)$ be the $q$-dimensional ball of radius $r$ in $x+E$, and define
\begin{equation}\label{eq:mean-operator}
\cM_{q,r}u(x):=
\inf_{E\in\Gr(q,n)}
\fintavg_{B_r^E(x)}u\,d\cH^q.
\end{equation}
If the eigenvalues of $A\in\Sym(n)$ are ordered increasingly, set
\[
\cP_q^-(A):=\lambda_1(A)+\cdots+\lambda_q(A).
\]
For $K\Subset\Omega$, let $\omega_{K,\phi}(r)$ denote the operator-norm modulus of continuity of $D^2\phi$ between points of distance at most $r$ in a fixed compact neighborhood of $K$.  In \Cref{sec:sectional-formulas} we prove the quantitative expansion
\begin{equation}\label{eq:mean-expansion-intro}
\left|\cM_{q,r}\phi(x)-\phi(x)
-
\frac{r^2}{2(q+2)}\cP_q^-(D^2\phi(x))\right|
\le
\frac{q}{2(q+2)}\omega_{K,\phi}(r)r^2
\end{equation}
uniformly on compact subsets for $\phi\in C^2$.  The infimum is $1$-Lipschitz with respect to the uniform norm.  Thus minimizing planes may jump at eigenvalue crossings, but the optimized value and the remainder remain stable.  On the convex branch,
\begin{equation}\label{eq:homogeneous-equivalence-intro}
\sigma_k(D^2\phi)=0
\quad\Longleftrightarrow\quad
\cP_q^-(D^2\phi)=0,
\end{equation}
so the zero-level equation has an asymptotic infimum mean-value property on $q$-dimensional affine sections.

Writing $\gamma_{k,n}$ for normalized Haar probability measure on $\Gr(k,n)$, we also record the complementary identity
\begin{equation}\label{eq:grassmann-intro}
\sigma_k(A)=\binom nk
\int_{\Gr(k,n)}\det(A|_V)\,d\gamma_{k,n}(V),
\end{equation}
which realizes the positive $k$-Hessian density as the Grassmannian average of the Monge--Amp\`ere densities of its $k$-dimensional restrictions.  Neither \eqref{eq:mean-expansion-intro} nor \eqref{eq:grassmann-intro} enters the proof of \Cref{thm:main}.  They are placed after the proof, in \Cref{sec:sectional-formulas}, together with an explanation of how the Newton-tensor/mean-value heuristic leads to the supporting-contact and John-ellipsoid mechanism used here.

\subsection{Organization}

In \Cref{sec:preliminaries} we fix the viscosity convention and state the external estimates used in the proof.  The proof of the main theorem then proceeds without interruption: small supporting contact sets imply regularity in \Cref{sec:contact-regularity}; the John-axis estimate is established in \Cref{sec:john}; section-volume decay and Hausdorff stratification are proved in \Cref{sec:section-decay,sec:main-proof}; and \Cref{sec:consequences} contains the $W^{2,1}$ consequence.  The logically independent sectional formulas are deliberately postponed to \Cref{sec:sectional-formulas}, where we also explain how the mean-value heuristic suggests the finite-scale contact-set proof.  Sharpness is discussed in \Cref{sec:sharpness}, and \Cref{sec:further} records limitations and further directions.  For completeness, \Cref{app:higher-codimension-covering} proves the higher-codimension section-covering criterion recorded in Mooney's remark, while \Cref{app:localized-chou-wang} verifies the localized Chou--Wang estimate used in the regularity argument.

\subsection*{Acknowledgments}
The author thanks Terry Tao for pointing out Reilly's Hessian identity \cite{Reilly1977} in response to an earlier MathOverflow question \cite{TaoMathOverflow2017}.  That observation was the starting point for the Newton-tensor and mean-value viewpoint developed in \Cref{sec:sectional-formulas}.  The author is grateful to Xushan Tu for carefully checking Sections~3--6 and for suggestions that improved the presentation and writing, and to Xi-Nan Ma for reading an early draft and offering helpful comments.

\subsection*{Statement on the use of artificial intelligence}
During the preparation of this manuscript, the author used large language models as auxiliary technical tools for literature discovery, brainstorming conceptual frameworks, and improving the presentation of mathematical arguments and English prose.  All mathematical claims and proofs were rigorously verified by the author.  The author directed the conceptual development and final presentation and assumes full responsibility for all results and for the contents of the manuscript.

\section{Preliminaries and external estimates}\label{sec:preliminaries}

\subsection{Algebraic facts}

The function
\[
F(A):=\sigma_k(A)^{1/k}
\]
is concave, positively one-homogeneous, and elliptic on $\Gamma_k$; see, for example, \cite{Wang2009}.  If $A\in\Gamma_k$, then the matrix
\[
\sigma_k^{ij}(A):=\frac{\partial\sigma_k}{\partial a_{ij}}(A)
\]
is positive definite.  Euler's identity gives
\begin{equation}\label{eq:euler}
DF(A):A=F(A),
\qquad
\sigma_k^{ij}(A)a_{ij}=k\sigma_k(A).
\end{equation}
Concavity and one-homogeneity also imply that, for $A,B\in\Gamma_k$,
\begin{equation}\label{eq:garding-linearized}
DF(A):B\ge F(B).
\end{equation}
Indeed,
\[
F(B)\le F(A)+DF(A):(B-A)=DF(A):B.
\]

We shall use a simple family of admissible saddle quadratics.

\begin{lemma}[Admissible saddle quadratic]\label{lem:saddle}
Let $\R^n=Y\oplus Z$ be an orthogonal decomposition with
\[
\dim Y=k,
\qquad
\dim Z=n-k.
\]
There exists $A_0=A_0(n,k)$ such that for every $A\ge A_0$,
\begin{equation}\label{eq:saddle-form}
Q(y,z):=A|y|^2-|z|^2
\end{equation}
satisfies $D^2Q\in\Gamma_k$.
\end{lemma}

\begin{proof}
For $1\le j\le k$,
\begin{align*}
2^{-j}\sigma_j(D^2Q)
&=
\sum_{\ell=0}^{j}
(-1)^\ell
\binom{n-k}{\ell}
\binom{k}{j-\ell}A^{j-\ell}.
\end{align*}
This is a polynomial in $A$ whose leading coefficient is $\binom{k}{j}>0$.  Since there are only finitely many $j\le k$, all these quantities are positive for $A$ sufficiently large.
\end{proof}

\begin{lemma}[Compactness inside the G\aa rding cone]\label{lem:compact-gamma}
For every $C_0<\infty$, the set
\[
\{A\in\Gamma_k:\sigma_k(A)=1,\ |A|\le C_0\}
\]
is contained in a compact subset of $\Gamma_k$.  Consequently, $F=\sigma_k^{1/k}$ is uniformly elliptic on this set, with constants depending only on $n,k,C_0$.
\end{lemma}

\begin{proof}
Maclaurin's inequalities give positive lower bounds for $\sigma_j(A)$, $1\le j<k$, in terms of $n,k$ and $\sigma_k(A)=1$.  Together with the norm bound, this prevents convergence to the boundary of $\Gamma_k$.  The uniform ellipticity follows from continuity and positivity of $DF$ on compact subsets of $\Gamma_k$.
\end{proof}

\subsection{Viscosity convention}

A $C^2$ function is called $k$-admissible if its Hessian belongs to $\Gamma_k$.  A continuous function $u$ is a viscosity subsolution of
\[
\sigma_k(D^2u)\ge f
\]
if every $k$-admissible $C^2$ function $\phi$ touching $u$ from above at $x$ satisfies
\[
\sigma_k(D^2\phi(x))\ge f(x).
\]
A viscosity supersolution is defined using $k$-admissible tests touching from below and the reversed inequality.  A viscosity solution is both a subsolution and a supersolution.  This is the convention used in \cite{MooneyRemarks2025}; the subsolution definition is unchanged if arbitrary $C^2$ upper tests are allowed \cite[Remark~2.2]{MooneyRemarks2025}.

We use the standard comparison and local-uniform stability properties of this admissible viscosity theory.  In particular, solutions on a ball can be approximated locally uniformly by smooth admissible solutions obtained from smooth approximations of the boundary values; see \cite[Theorem~2.1 and Remark~2.3]{MooneyRemarks2025} and the original Dirichlet theory \cite{CNS1985}.

\subsection{Classical solvability, gradient bounds, and the Chou--Wang estimate}

We record the precise inputs needed for the regularity criterion.

\begin{theorem}[Caffarelli--Nirenberg--Spruck]\label{thm:CNS}
Let $B_R\subset\R^n$, let $g\in C^\infty(\partial B_R)$, and let $1\le k\le n$.  There exists a unique solution
\[
v\in C^\infty(\ol B_R),
\qquad
D^2v\in\Gamma_k,
\]
of
\[
\sigma_k(D^2v)=1\quad\text{in }B_R,
\qquad
v=g\quad\text{on }\partial B_R.
\]
\end{theorem}

This is the ball case of \cite{CNS1985}; the same formulation is stated in \cite[Theorem~2.1]{MooneyRemarks2025}.

\begin{theorem}[Trudinger's interior gradient estimate]\label{thm:gradient}
If $v$ is a smooth $k$-admissible solution of $\sigma_k(D^2v)=1$ in $B_R$, then
\begin{equation}\label{eq:gradient-estimate}
\|Dv\|_{L^\infty(B_{R/2})}
\le C(n,k)R^{-1}\|v\|_{L^\infty(B_R)}.
\end{equation}
\end{theorem}

See \cite{Trudinger1997} and \cite[Theorem~2.4]{MooneyRemarks2025}.

\begin{theorem}[Localized Chou--Wang estimate]\label{thm:Chou-Wang}
Let $D\subset\R^n$ be bounded and open.  Suppose that
\[
v,w\in C^\infty(D)\cap C(\overline D),
\qquad D^2v\in\Gamma_k,
\qquad \sigma_k(D^2v)=1\quad\text{in }D.
\]
Let
\[
L_v\phi:=\sigma_k^{ij}(D^2v)\phi_{ij},
\]
and assume that
\[
w>v\quad\text{in }D,
\qquad w=v\quad\text{on }\partial D,
\qquad L_vw\ge0\quad\text{in }D.
\]
If $Dv,Dw\in L^\infty(D)$, then
\begin{equation}\label{eq:Chou-Wang-estimate}
\sup_D (w-v)^4|D^2v|
\le
C\bigl(n,k,\|Dv\|_{L^\infty(D)},\|Dw\|_{L^\infty(D)}\bigr).
\end{equation}
\end{theorem}

Chou and Wang prove the estimate with a $k$-admissible comparison function
\cite[Theorem~4.1]{ChouWang2001}.  Mooney records the bounded-domain form above
and observes that the weaker condition $L_vw\ge0$ suffices
\cite[Theorem~2.5 and Remark~2.6]{MooneyRemarks2025}.  Since those localization
and weakening steps are stated there without proof, we verify them in
\Cref{app:localized-chou-wang}.  The comparison function used in the main
argument is itself $k$-admissible, so the proof also falls within the original
admissible-barrier formulation.

\subsection{Mooney's convex section-covering theorem}

For a convex function $v:\Omega\to\R$, a point $x\in\Omega$, and $p\in\partial v(x)$, define the section
\begin{equation}\label{eq:section-definition}
S^v_{h,p}(x):=
\{y\in\Omega:v(y)<v(x)+p\cdot(y-x)+h\}.
\end{equation}

\begin{lemma}[Local supporting slopes are global]\label{lem:local-global-subgradient}
Let $\Omega$ be convex, let $v:\Omega\to\R$ be convex, and let $x\in\Omega$.  If an affine function with slope $p$ supports $v$ at $x$ on some neighborhood of $x$, then $p\in\partial v(x)$ relative to all of $\Omega$.
\end{lemma}

\begin{proof}
Let $z\in\Omega$.  For all sufficiently small $t>0$, the point
$y_t=x+t(z-x)$ lies in the neighborhood on which the supporting inequality holds.  Hence
\[
v(x)+t p\cdot(z-x)\le v(y_t)
\le (1-t)v(x)+t v(z),
\]
where the second inequality is convexity.  Dividing by $t$ gives
$v(z)\ge v(x)+p\cdot(z-x)$.
\end{proof}

\begin{lemma}[Quadratic shift of subgradients]\label{lem:quadratic-subgradient-shift}
Let $\Omega\subset\R^n$ be open and convex, let $u:\Omega\to\R$ be convex,
and set
\[
v(y):=u(y)+\frac12|y|^2.
\]
Then, for every $x\in\Omega$,
\begin{equation}\label{eq:quadratic-subgradient-shift}
\partial v(x)=\partial u(x)+x.
\end{equation}
Equivalently, $p^v\in\partial v(x)$ if and only if $p^v-x\in\partial u(x)$.
\end{lemma}

\begin{proof}
If $p\in\partial u(x)$, then for every $y\in\Omega$,
\[
v(y)-v(x)-(p+x)\cdot(y-x)
=
u(y)-u(x)-p\cdot(y-x)+\frac12|y-x|^2\ge0,
\]
so $p+x\in\partial v(x)$.

Conversely, let $p^v\in\partial v(x)$, fix $y\in\Omega$, and put $d=y-x$.
For $0<t\le1$, convexity of $\Omega$ and the supporting inequality for $v$
give
\[
\frac{u(x+td)-u(x)}{t}
\ge (p^v-x)\cdot d-\frac{t}{2}|d|^2.
\]
Letting $t\downarrow0$ yields
\[
u'(x;d)\ge(p^v-x)\cdot d,
\]
where $u'(x;d)$ is the one-sided directional derivative.  Convexity on the
segment $[x,y]$ gives $u(y)-u(x)\ge u'(x;d)$, and hence
\[
u(y)\ge u(x)+(p^v-x)\cdot(y-x).
\]
Thus $p^v-x\in\partial u(x)$.
\end{proof}

\begin{theorem}[Higher-codimension section-covering criterion]\label{thm:Mooney}
Let $v$ be convex in $B_1\subset\R^n$, let $1\le s\le n-1$ be an integer, and let $E\subset B_1$.  Suppose that for every $x\in E$ and every $p\in\partial v(x)$ there are constants $C_{x,p}<\infty$ and $h_{x,p}>0$ such that
\begin{equation}\label{eq:Mooney-hypothesis}
|S^v_{h,p}(x)|
\le C_{x,p}h^{(n+s)/2}
\qquad(0<h<h_{x,p}).
\end{equation}
Then
\begin{equation}\label{eq:Mooney-conclusion}
\cH^{n-s}(E)=0.
\end{equation}
\end{theorem}

The case $s=1$ is \cite[Theorem~3.1]{Mooney2015}; the higher-codimension statement is recorded in \cite[Remark~3.4]{Mooney2015}.  Because the latter is stated there without proof, a complete argument is included in \Cref{app:higher-codimension-covering}.  The constants and height thresholds may depend on both the point and the supporting slope.

\section{Small contact sets imply regularity}\label{sec:contact-regularity}

We first make explicit the stabilization implicit in \eqref{eq:contact-dimension}.

\begin{lemma}[Stabilization of the local affine hull]\label{lem:stabilization}
Let $u$ be convex, $x\in\Omega$, and $p\in\partial u(x)$.  There exist $r_0>0$ and an affine subspace $Z_{x,p}$ through $x$ such that
\[
\aff K_{x,p}(r)=Z_{x,p}
\qquad(0<r<r_0),
\]
and $\dim Z_{x,p}=d_u(x,p)$.
\end{lemma}

\begin{proof}
The sets $K_{x,p}(r)$ are nested, contain $x$, and their affine dimensions are nonincreasing integers.  Once the dimensions stabilize, the affine hulls form a nested family of affine subspaces through $x$ of equal dimension, hence are equal.
\end{proof}

\begin{lemma}[Smooth approximation on an interior ball]\label{lem:smooth-approximation}
Let $\ol B_{2R}\Subset\Omega$, and let $u\in C(\Omega)$ be a viscosity solution of \eqref{eq:main-equation} in a neighborhood of $\ol B_{2R}$.  There exist smooth $k$-admissible solutions
$u_j\in C^\infty(\ol B_{2R})$ of
\[
\sigma_k(D^2u_j)=1\quad\text{in }B_{2R}
\]
such that $u_j\to u$ uniformly on $\ol B_{2R}$.  Moreover, for every $R'<2R$,
\begin{equation}\label{eq:uniform-gradient-approximants}
\sup_j\|Du_j\|_{L^\infty(B_{R'})}<\infty.
\end{equation}
\end{lemma}

\begin{proof}
Choose $g_j\in C^\infty(\partial B_{2R})$ with
$g_j\to u|_{\partial B_{2R}}$ uniformly, and set
\[
\eps_j:=\|g_j-u\|_{L^\infty(\partial B_{2R})}\longrightarrow0.
\]
Let $u_j$ be the solutions furnished by \Cref{thm:CNS}.  Since vertical translations do not change the equation, $u-\eps_j$ and $u+\eps_j$ are respectively a viscosity subsolution and supersolution with
\[
u-\eps_j\le g_j\le u+\eps_j
\quad\text{on }\partial B_{2R}.
\]
The comparison principle therefore gives
\[
\|u_j-u\|_{L^\infty(\ol B_{2R})}\le\eps_j.
\]
This is also the approximation asserted in \cite[Remark~2.3]{MooneyRemarks2025}.  In particular, the $L^\infty$ norms of $u_j$ are uniformly bounded.

Fix $R'<2R$ and put $d=2R-R'>0$.  Every point of $B_{R'}$ is the center of a ball of radius $d/2$ contained in $B_{2R}$.  Applying \Cref{thm:gradient}, after translation, on these balls yields
\[
\|Du_j\|_{L^\infty(B_{R'})}
\le C(n,k)d^{-1}\|u_j\|_{L^\infty(B_{2R})},
\]
which is uniform in $j$.
\end{proof}

\begin{remark}\label{rem:approximants-not-convex}
For $k<n$, the Caffarelli--Nirenberg--Spruck approximants need not be convex; they are only $k$-admissible.  This is sufficient here.  The localized Chou--Wang estimate is formulated for $k$-admissible solutions.  In our application the barrier is itself $k$-admissible (and therefore also satisfies the required linearized inequality), so no convexity of the approximating sequence is needed.
\end{remark}

\begin{remark}[Where the constant right-hand side is used]\label{rem:constant-rhs-approximation}
The proof of \Cref{lem:smooth-approximation} uses the constant right-hand side in two separate ways.  First, $u_j$ and $u$ solve the same equation, so vertical shifts of $u$ are still solutions and comparison immediately converts convergence of the boundary data into uniform convergence on the closed ball.  Second, the interior gradient estimate used for the approximants has a constant depending only on $n$, $k$, the radius, and the $L^\infty$ norm.  For a variable right-hand side $f_j(x)$, neither simplification is automatic: vertical shifts do not compare solutions with different $f_j$, and the known gradient estimates involve derivatives of $f_j^{1/k}$.  The precise consequences for rough data are discussed in \Cref{sec:variable-rhs}.
\end{remark}

\begin{lemma}[Barrier regularity]\label{lem:barrier-regularity}
Let $u$ be a convex viscosity solution of \eqref{eq:main-equation} in a neighborhood of $0$.  Suppose
\begin{equation}\label{eq:barrier-normalization}
u(0)=0,
\qquad
u\ge0,
\end{equation}
and that for some $r>0$ with $B_{2r}\Subset\Omega$,
\begin{equation}\label{eq:contact-in-Z}
\{u=0\}\cap B_{2r}\subset Z
\end{equation}
for a linear subspace $Z$ of dimension at most $n-k$.  Then $u$ is smooth in a neighborhood of $0$.
\end{lemma}

\begin{proof}
Enlarge $Z$ if necessary so that $\dim Z=n-k$, and put $Y=Z^\perp$, so that $\dim Y=k$.  Write $x=y+z$ according to the orthogonal splitting $Y\oplus Z$.  Choose $A\ge A_0(n,k)$ as in \Cref{lem:saddle}, fix $\alpha\in(0,1)$, and set
\begin{equation}\label{eq:barrier-Q}
Q(y,z):=A|y|^2-|z|^2+\alpha r^2.
\end{equation}
By \Cref{lem:saddle}, $D^2Q\in\Gamma_k$.

On $\partial B_r$, the inequality $Q\ge0$ implies
\[
(A+1)|y|^2\ge(1-\alpha)r^2.
\]
Consequently, the compact set
\[
F:=\{x\in\partial B_r:Q(x)\ge0\}
\]
has positive distance from $Z$.  By \eqref{eq:contact-in-Z}, $u>0$ on $F$.  We may therefore choose $\delta>0$ sufficiently small that the $k$-admissible function
\begin{equation}\label{eq:w-barrier}
w:=\delta Q
\end{equation}
satisfies
\begin{equation}\label{eq:barrier-gap}
w<u\quad\text{on }\partial B_r,
\qquad
w(0)>u(0)=0.
\end{equation}
Indeed, on $\partial B_r\setminus F$ one has $w<0\le u$, while compactness gives a positive lower bound for $u$ on $F$.

Let $u_j$ be the smooth approximations of $u$ furnished by \Cref{lem:smooth-approximation} on $B_{2r}$.  Uniform convergence and \eqref{eq:barrier-gap} give constants $c_0,\eta,\rho>0$ such that, for every sufficiently large $j$,
\begin{equation}\label{eq:uniform-gaps}
w-u_j\ge c_0\quad\text{on }B_\rho,
\qquad
w-u_j\le-\eta\quad\text{on }\partial B_r.
\end{equation}
Choose a regular value $t_j\in(0,c_0/2)$ of $w-u_j$, and let $D_j$ be the connected component containing $0$ of
\[
\{w-u_j>t_j\}\cap B_r.
\]
The boundary gap in \eqref{eq:uniform-gaps} prevents $\ol D_j$ from meeting $\partial B_r$.  Since $t_j$ is a regular value, $\partial D_j$ is a smooth compact level hypersurface.  Moreover, $B_\rho$ is connected, contains $0$, and lies in the same superlevel set; hence $B_\rho\subset D_j$.  Thus $D_j\Subset B_r$ is a bounded smooth domain and
\begin{equation}\label{eq:CW-boundary-data}
u_j<w-t_j\quad\text{in }D_j,
\qquad
u_j=w-t_j\quad\text{on }\partial D_j.
\end{equation}

The function $w-t_j$ has the same Hessian as the $k$-admissible quadratic $w$, and hence it is an admissible Chou--Wang barrier on $D_j$.  Thus \Cref{thm:Chou-Wang} applies directly.  The gradients of $u_j$ are uniformly bounded on $B_r$ by \Cref{lem:smooth-approximation}, and $w$ is fixed.  Since
\[
w-t_j-u_j\ge c_0/2\quad\text{on }B_\rho,
\]
we obtain
\begin{equation}\label{eq:uniform-Hessian-bound}
\sup_{B_\rho}|D^2u_j|\le C.
\end{equation}

By \Cref{lem:compact-gamma}, the matrices $D^2u_j(B_\rho)$ lie in a fixed compact subset of $\Gamma_k$.  Since $\Gamma_k$ is convex, the convex hull of this compact set is still compactly contained in $\Gamma_k$; hence $F$ is concave and uniformly elliptic on a fixed neighborhood of all Hessian values under consideration.  The local Evans--Krylov theorem (equivalently, after a standard concave uniformly elliptic extension of $F$) and Schauder estimates \cite[Chapter~6]{CaffarelliCabre1995} give uniform higher-derivative bounds on a smaller ball.  Passing to the limit proves that $u$ is smooth near $0$.
\end{proof}

\begin{proof}[Proof of \Cref{thm:contact-regularity}]
Subtract the supporting affine function with slope $p_0$ and translate $x_0$ to the origin.  Then $u(0)=0$ and $u\ge0$.  By \Cref{lem:stabilization} and \eqref{eq:small-contact-intro}, for some sufficiently small $r$ the contact set in $B_{2r}$ is contained in an affine subspace through the origin of dimension at most $n-k$.  Apply \Cref{lem:barrier-regularity}.

For the final assertion, if $x\in\Sing(u)$ and there were a slope $p\in\partial u(x)$ with $d_u(x,p)\le n-k$, the first part would make $x$ regular.  Thus $d_u(x,p)\ge n-k+1$ for every supporting slope, proving \eqref{eq:sing-in-Cq}.
\end{proof}

\section{Positive \texorpdfstring{$k$}{k}-Hessian density and John axes}\label{sec:john}

The central finite-scale estimate controls the smallest $k$ semiaxes of a sublevel body.  We use John's theorem in the form: if $K\subset\R^n$ is a compact convex body with nonempty interior, then there exist $z\in\R^n$ and a positive definite symmetric matrix $A$ such that
\begin{equation}\label{eq:John-general}
z+AB_1\subset K\subset z+nAB_1.
\end{equation}

\begin{lemma}[Product of the smallest $k$ John axes]\label{lem:axis-product}
Let $U\subset\R^n$ be open with $\ol B_R\Subset U$, and let $u\in C(U)$ be convex in $U$ and satisfy
\begin{equation}\label{eq:axis-normalization}
u(0)=0,
\qquad
u\ge0\quad\text{in }U.
\end{equation}
Assume in the viscosity sense in $U$ that
\begin{equation}\label{eq:axis-subsolution}
\sigma_k(D^2u)\ge\lambda>0.
\end{equation}
For sufficiently small $h>0$, set
\[
K_h:=\{u\le h\}\cap\ol B_R,
\]
and choose a John ellipsoid
\[
E_h=z_h+A_hB_1,
\qquad
E_h\subset K_h\subset z_h+nA_hB_1,
\]
where $A_h$ is symmetric and positive definite.  Denote its semiaxes by
\[
a_1(h)\ge\cdots\ge a_n(h)>0.
\]
If $q=n-k+1$, then
\begin{equation}\label{eq:axis-product}
a_q(h)a_{q+1}(h)\cdots a_n(h)
\le 2^{k/2}\binom nk^{1/2}\lambda^{-1/2}h^{k/2}.
\end{equation}
\end{lemma}

\begin{proof}
Continuity on a neighborhood of $\ol B_R$ makes $K_h$ closed in the compact ball, while convexity makes it convex.  Since $u(0)=0<h$, continuity also gives a Euclidean neighborhood of $0$ contained in $K_h$.  Thus $K_h$ is a compact convex body with nonempty interior.  Consider the ellipsoidal paraboloid
\begin{equation}\label{eq:John-paraboloid}
P_h(x):=h\,|A_h^{-1}(x-z_h)|^2.
\end{equation}
Every point of $\partial K_h$ lies outside the interior of $E_h$, and hence
\[
P_h\ge h\ge u
\quad\text{on }\partial K_h.
\]
At the center, $P_h(z_h)=0\le u(z_h)$.  Therefore
\[
\mu_h:=\min_{K_h}(P_h-u)\le0.
\]
On $\partial K_h$ one has $P_h-u\ge0$.  If $\mu_h<0$, every minimizer is consequently in $\operatorname{int}K_h$.  If $\mu_h=0$, then
$0=P_h(z_h)-u(z_h)=-u(z_h)$, so the center $z_h$ itself is a minimizer; since
$z_h\in E_h\subset\operatorname{int}K_h$, it is again an interior contact point.  In either case this point lies in $B_R\Subset U$, and $P_h-\mu_h$ touches $u$ from above there.  Since $D^2P_h=2hA_h^{-2}>0$, this is an admissible upper test.  The viscosity inequality yields
\begin{equation}\label{eq:axis-test}
\lambda
\le\sigma_k(D^2P_h)
=(2h)^k\sigma_k(A_h^{-2}).
\end{equation}
The eigenvalues of $A_h^{-2}$ are $a_i^{-2}$.  Since $a_1\ge\cdots\ge a_n$, every $k$-fold product occurring in $\sigma_k(A_h^{-2})$ is bounded above by
\[
(a_q\cdots a_n)^{-2}.
\]
Thus
\[
\sigma_k(A_h^{-2})
\le\binom nk(a_q\cdots a_n)^{-2}.
\]
Combining this estimate with \eqref{eq:axis-test} gives \eqref{eq:axis-product}.
\end{proof}

\begin{figure}[t]
  \centering
  \includegraphics[width=\textwidth]{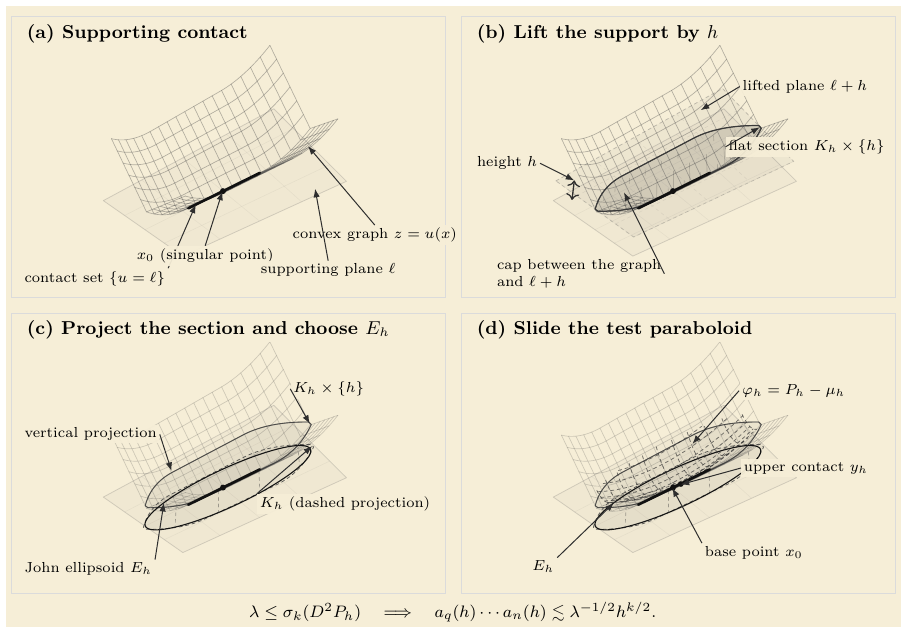}
  \caption{The four geometric steps in the interior viscosity test of
  \Cref{lem:axis-product}, shown in one fixed three-dimensional projection.
  \textup{(a)} A supporting plane $\ell$ meets the graph along a nontrivial contact piece.  \textup{(b)} Lifting $\ell$ by $h$ exposes a graph cap whose projection is the variable-space section $K_h$.  \textup{(c)} In the variable space, John's theorem gives
  $E_h=z_h+A_hB_1\subset K_h\subset z_h+nA_hB_1$.
  \textup{(d)} The ellipsoidal paraboloid determined by $E_h$ is translated until $P_h-\mu_h$ touches $u$ from above at an interior point $y_h$.  The viscosity inequality then controls the product of the smallest $k$ semiaxes.}
  \label{fig:john-viscosity-test}
\end{figure}

The four stages of the argument are summarized in \Cref{fig:john-viscosity-test}.

We need two elementary geometric consequences of ellipsoidal containment.

\begin{lemma}[A relative ball controls a John axis]\label{lem:relative-ball-axis}
Let $1\le d\le n$.  Let $K\subset\R^n$ be convex and suppose
\[
K\subset z+CAB_1,
\]
where $C\ge1$, $A$ is symmetric positive definite, and the semiaxes of $A$ are ordered as $a_1\ge\cdots\ge a_n$.  If $K$ contains a $d$-dimensional Euclidean ball of radius $\rho$ in an affine $d$-plane, then
\begin{equation}\label{eq:relative-ball-axis}
a_d\ge\rho/C.
\end{equation}
\end{lemma}

\begin{proof}
Let $L$ be the direction space of the affine $d$-plane.  The width of the relative ball in every unit direction $e\in L$ equals $2\rho$, whereas the width of $z+CAB_1$ in direction $e$ equals $2C|Ae|$.  Hence
\[
\min_{e\in L,\,|e|=1}|Ae|\ge\rho/C.
\]
The min--max characterization of the singular values gives
\[
a_d=\max_{\dim W=d}\min_{e\in W,\,|e|=1}|Ae|,
\]
which proves the claim.
\end{proof}

\begin{lemma}[Ellipsoid--ball intersection]\label{lem:ellipsoid-ball}
Let $A$ be symmetric positive definite with semiaxes $a_1\ge\cdots\ge a_n$.  For all $z,x\in\R^n$, $C\ge1$, and $r>0$,
\begin{equation}\label{eq:ellipsoid-ball-intersection}
|(z+CAB_1)\cap B_r(x)|
\le C_n C^n\prod_{i=1}^n\min\{a_i,r\}.
\end{equation}
\end{lemma}

\begin{proof}
Rotate to principal-axis coordinates for $A$.  The ellipsoid is contained in a rectangular box with side lengths $2Ca_i$, while the ball is contained in an axis-parallel box with side lengths $2r$.  Their intersection is contained in a box whose $i$th side has length at most $2\min\{Ca_i,r\}\le2C\min\{a_i,r\}$.  Taking volumes proves the estimate.
\end{proof}

\section{Section-volume decay at flat supporting contacts}\label{sec:section-decay}

We now turn contact dimension into a quantitative section-volume bound.  First observe that a supporting contact set is convex: if $u\ge\ell$ and equality holds at two points, convexity forces equality on the segment joining them.

\begin{proposition}[Quantitative section decay at a flat contact]\label{prop:section-decay}
Let $u$ be convex in $\Omega$ and satisfy \eqref{eq:positive-subsolution}.  Fix $x_0\in\Omega$, $p_0\in\partial u(x_0)$, and an integer
\[
q=n-k+1\le m\le n-1.
\]
Assume that, after translating $x_0$ to the origin and subtracting the supporting affine function with slope $p_0$, there are numbers $R>0$, $\rho>0$, an affine $m$-plane $a+L$, and a relative ball
\begin{equation}\label{eq:quantitative-contact-ball}
a+B_\rho^L(0)
\subset
\{u=0\}\cap B_{R/2},
\qquad B_R\Subset\Omega.
\end{equation}
Define
\begin{equation}\label{eq:v-perturbation}
v(x):=u(x)+\frac12|x|^2,
\qquad
p_0^v:=p_0+x_0\in\partial v(x_0).
\end{equation}
Then there is $c_n>0$ such that, for
\begin{equation}\label{eq:quantitative-height-range}
0<h<c_n\min\{R^2,\rho^2\},
\end{equation}
one has
\begin{equation}\label{eq:section-decay-quantitative}
|S^v_{h,p_0^v}(x_0)|
\le
C(n,k,m)\lambda^{-1/2}
\rho^{-(m-q+1)}h^{(m+k)/2}.
\end{equation}
Consequently, if
\begin{equation}\label{eq:contact-dim-m}
d_u(x_0,p_0)\ge m,
\end{equation}
then there exist constants $C_{x_0,p_0}<\infty$ and $h_{x_0,p_0}>0$ such that
\begin{equation}\label{eq:section-decay}
|S^v_{h,p_0^v}(x_0)|
\le C_{x_0,p_0}\lambda^{-1/2}h^{(m+k)/2}
\qquad(0<h<h_{x_0,p_0}).
\end{equation}
\end{proposition}

\begin{proof}
It suffices to prove the quantitative assertion under the normalization
\begin{equation}\label{eq:section-normalization}
x_0=0,
\qquad p_0=0,
\qquad u(0)=0,
\qquad u\ge0.
\end{equation}
For $h>0$, apply \Cref{lem:axis-product} to
\[
K_h=\{u\le h\}\cap\ol B_R
\]
and write the John inclusions as
\[
z_h+A_hB_1\subset K_h\subset z_h+nA_hB_1,
\]
with semiaxes
\[
a_1(h)\ge\cdots\ge a_n(h)>0.
\]
The contact ball in \eqref{eq:quantitative-contact-ball} is contained in $K_h$.  By \Cref{lem:relative-ball-axis}, applied to the upper John inclusion,
\begin{equation}\label{eq:a_m-lower-explicit}
a_m(h)\ge \rho/n.
\end{equation}
Since $m\ge q$ and the axes are decreasing,
\[
a_q(h)\cdots a_m(h)
\ge (\rho/n)^{m-q+1}.
\]
Dividing the smallest-axis product estimate \eqref{eq:axis-product} by this lower bound gives
\begin{equation}\label{eq:remaining-axes-explicit}
a_{m+1}(h)\cdots a_n(h)
\le
C(n,k,m)\lambda^{-1/2}
\rho^{-(m-q+1)}h^{k/2}.
\end{equation}

Under the normalization \eqref{eq:section-normalization}, the relevant section is
\begin{equation}\label{eq:section-v-normalized}
S^v_{h,0}(0)
=
\{y:u(y)+\tfrac12|y|^2<h\}.
\end{equation}
If $h<R^2/8$, then
\begin{equation}\label{eq:section-in-intersection}
S^v_{h,0}(0)
\subset
\{u<h\}\cap B_{\sqrt{2h}}(0)
\subset
(z_h+nA_hB_1)\cap B_{\sqrt{2h}}(0).
\end{equation}
Choose $c_n$ small enough that \eqref{eq:quantitative-height-range} implies
\[
\sqrt{2h}\le \frac{\rho}{2n}\le \frac12a_m(h).
\]
Then $a_i(h)>\sqrt{2h}$ for $1\le i\le m$.  Applying \Cref{lem:ellipsoid-ball} and then \eqref{eq:remaining-axes-explicit},
\begin{align*}
|S^v_{h,0}(0)|
&\le C_n\prod_{i=1}^n\min\{a_i(h),\sqrt{2h}\}\\
&\le C(n,m)h^{m/2}a_{m+1}(h)\cdots a_n(h)\\
&\le C(n,k,m)\lambda^{-1/2}
\rho^{-(m-q+1)}h^{(m+k)/2}.
\end{align*}
This proves \eqref{eq:section-decay-quantitative}.  Undoing the translation and affine subtraction is justified by the identity
\[
v(y)-v(x_0)-p_0^v\cdot(y-x_0)
=
u(y)-u(x_0)-p_0\cdot(y-x_0)+\frac12|y-x_0|^2.
\]

If \eqref{eq:contact-dim-m} holds, then for some sufficiently small $R$ the local contact set has affine dimension at least $m$.  Since it is convex, it contains an $m$-dimensional relative ball of some radius $\rho_{x_0,p_0}>0$ in $B_{R/2}$.  Applying the quantitative assertion with this radius yields \eqref{eq:section-decay}.
\end{proof}

\begin{figure}[t]
  \centering
  \includegraphics[width=\textwidth]{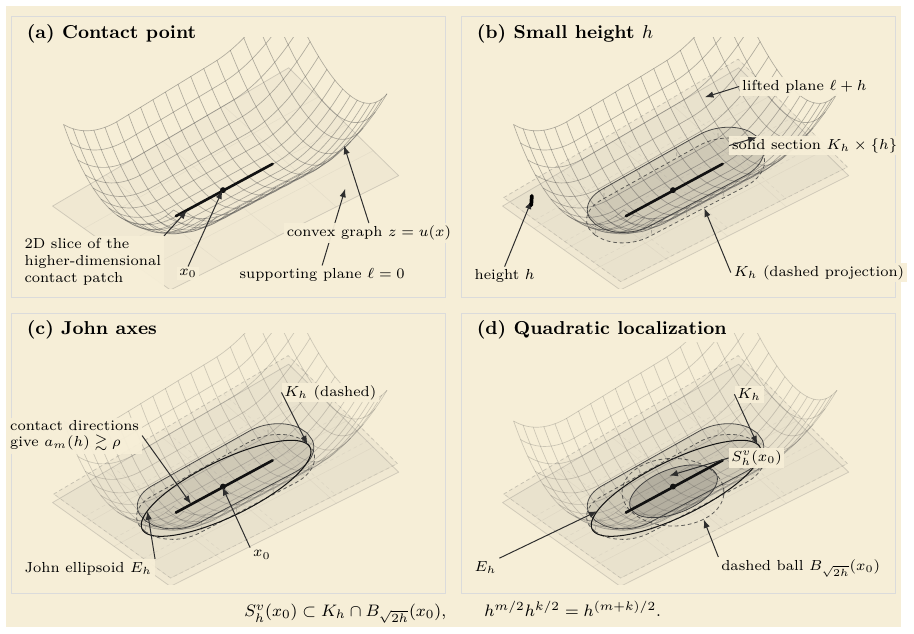}
  \caption{The four geometric stages of the localization mechanism in
  \Cref{prop:section-decay}, shown in one fixed three-dimensional projection.
  \textup{(a)} The supporting plane contains a contact piece of fixed width
  $\rho$ in $m$ directions.  \textup{(b)} At height $h$, its sublevel body is
  $K_h$.  \textup{(c)} Width comparison and the singular-value min--max
  principle give $a_m(h)\gtrsim\rho$ for a John ellipsoid $E_h$; the contact
  directions need not coincide with its principal axes.  \textup{(d)} For
  $v=u+\tfrac12|x|^2$, one has
  $S_h^v(x_0)\subset K_h\cap B_{\sqrt{2h}}(x_0)$, which yields the factor
  $h^{m/2}\cdot h^{k/2}=h^{(m+k)/2}$.}
  \label{fig:section-localization}
\end{figure}

The geometric inputs and the two sources of the exponent in \eqref{eq:section-decay-quantitative} are displayed in \Cref{fig:section-localization}.

\begin{remark}[What is uniform in the preceding estimate]\label{rem:uniform-contact-stratum}
For fixed $R$ and $\rho$, the constant and the height range in \eqref{eq:section-decay-quantitative} are uniform over all points and supporting slopes whose contact sets contain an $m$-dimensional relative ball of radius $\rho$ inside $B_{R/2}$.  More precisely, on any class of point--slope pairs $(x,p)$ for which
\[
B_R(x)\Subset\Omega
\quad\text{and}\quad
\{u=\ell_{x,p}\}\cap B_{R/2}(x)
\text{ contains an $m$-ball of radius at least $\rho$},
\]
the estimate
\[
|S^{u+|\cdot|^2/2}_{h,p+x}(x)|
\le C(n,k,m)\lambda^{-1/2}\rho^{-(m-q+1)}h^{(m+k)/2}
\]
holds with the same constant for every
$0<h<c_n\min\{R^2,\rho^2\}$.  Thus the proof supplies a genuinely quantitative estimate on each uniformly nondegenerate contact class.  The passage from contact dimension to such a class is not uniform: the radius $\rho_{x,p}$ may tend to zero arbitrarily fast along a singular sequence or as the supporting slope varies.
\end{remark}

\section{Hausdorff stratification and proof of the main theorem}\label{sec:main-proof}

\begin{proof}[Proof of \Cref{thm:stratification}]
After a translation, fix concentric balls $B_R\Subset B_{2R}\Subset\Omega$ and set
\[
v(x):=u(x)+\tfrac12|x|^2.
\]
We apply \Cref{thm:Mooney} to the restriction of $v$ to $B_{2R}$.  If
$x\in B_R$ and $p^v$ is a supporting slope for this restriction at $x$, then
\Cref{lem:local-global-subgradient} shows that $p^v\in\partial v(x)$ relative to all of $\Omega$.  By \Cref{lem:quadratic-subgradient-shift},
\[
p:=p^v-x\in\partial u(x).
\]
For $x\in\cC_m(u)\cap B_R$ one has $d_u(x,p)\ge m$, and \Cref{prop:section-decay} gives
\[
|S^v_{h,p^v}(x)|
\le C_{x,p^v}h^{(m+k)/2}
\]
for all sufficiently small $h$.  Moreover,
\[
v(y)-v(x)-p^v\cdot(y-x)
\ge \frac12|y-x|^2,
\]
so $S^v_{h,p^v}(x)\subset B_{\sqrt{2h}}(x)$.  Thus, after decreasing the point-dependent height threshold, the section is contained in $B_{2R}$ and agrees with the section of the restricted function.  Put
\begin{equation}\label{eq:s-def}
s:=m+k-n.
\end{equation}
The assumption $m\ge n-k+1$ gives $s\ge1$, while $m\le n-1$ gives $s\le k-1\le n-1$.  Since
\[
\frac{m+k}{2}=\frac{n+s}{2},
\]
translating $B_{2R}$ to the unit ball and rescaling space multiplies section volumes and Hausdorff measures only by fixed powers of $R$; it does not change the exponent.  Hence \Cref{thm:Mooney} gives
\[
\cH^{n-s}\bigl(\cC_m(u)\cap B_R\bigr)=0.
\]
Because $n-s=2n-m-k$, this is \eqref{eq:strata-bound}.  Exhausting $\Omega$ by interior balls proves the local assertion.

Finally, suppose $x\in\cC_n(u)$.  For any $p\in\partial u(x)$, the stabilized local contact set has full affine dimension and therefore contains an open ball on which $u$ agrees with $\ell_{x,p}$.  At an interior point $y$ of that ball, the function
\[
\ell_{x,p}(z)+\eps|z-y|^2
\]
touches $u$ from above.  Choosing $\eps>0$ so small that $\sigma_k(2\eps I)<\lambda$ contradicts \eqref{eq:positive-subsolution}.  Hence $\cC_n(u)=\varnothing$.
\end{proof}

\begin{proof}[Proof of \Cref{thm:main}]
By \Cref{thm:contact-regularity},
\[
\Sing(u)\subset\cC_q(u),
\qquad q=n-k+1.
\]
Applying \Cref{thm:stratification} with $m=q$ and using $2n-q-k=n-1$ gives
\[
\cH^{n-1}_{\mathrm{loc}}(\Sing(u))=0.
\]
At every point of $\Reg(u)$, the equation is uniformly elliptic in a sufficiently small neighborhood after the local Hessian bound, and the standard Evans--Krylov and Schauder estimates \cite[Chapter~6]{CaffarelliCabre1995} imply smoothness.
\end{proof}

\section{Consequences}\label{sec:consequences}

We prove the Sobolev consequence stated in the introduction.  The argument is the same measure-theoretic mechanism used by Mooney for singular Monge--Amp\`ere solutions: the distributional Hessian of a convex function cannot concentrate on a set of vanishing $\cH^{n-1}$ measure \cite{Mooney2015}.  This is a qualitative absolute-continuity argument, rather than an affine-normalized higher-integrability estimate of the type developed in \cite{DePhilippisFigalli2013,DePhilippisFigalliSavin2013,Schmidt2013}; in particular, it supplies no quantitative $L\log L$ modulus and no $L^{1+\varepsilon}$ gain for the Hessian density.

\begin{lemma}[A codimension-one growth bound for the Hessian measure]\label{lem:Hessian-measure-growth}
Let $U\subset\R^n$ be open and convex, and let $u$ be convex in $U$.  Then $u$ is locally Lipschitz and its distributional Hessian $D^2u$ is a positive semidefinite symmetric matrix-valued Radon measure.  Indeed, for every $\xi\in\R^n$, convexity on lines implies that the distribution $\partial_{\xi\xi}u$ is positive, hence is a Radon measure; polarization then shows that every $\partial_{ij}u$ is a signed Radon measure.  This is also a standard consequence of the distributional characterization of convexity; see, for example, \cite[Chapter~6]{EvansGariepy2015}.  If $B_{2r}(x)\Subset U$, then
\begin{equation}\label{eq:Laplacian-growth}
\Delta u(B_r(x))
\le C_n\operatorname{Lip}(u;B_{2r}(x))\,r^{n-1}.
\end{equation}
Consequently, $|D^2u|$ vanishes on every Borel set of zero $\cH^{n-1}$ measure.
\end{lemma}

\begin{proof}
For completeness, if $\varphi\in C_c^\infty(U)$ is nonnegative and $\xi\in\R^n$ is a unit vector, set
\[
I_y:=\{t\in\R:y+t\xi\in U\},\qquad y\in\xi^\perp.
\]
Because $U$ is convex, each nonempty $I_y$ is an interval.  The restriction
$t\mapsto u(y+t\xi)$ is convex on $I_y$, while
$t\mapsto\varphi(y+t\xi)$ has compact support in $I_y$.  Fubini's theorem
and the one-dimensional distributional characterization of convexity therefore give
\[
\langle\partial_{\xi\xi}u,\varphi\rangle
=\int_{\xi^\perp}\left(\int_{I_y}u(y+t\xi)\,\frac{d^2}{dt^2}\varphi(y+t\xi)\,dt\right)d\cH^{n-1}(y)\ge0,
\]
where empty fibers contribute zero.
Thus $\partial_{\xi\xi}u$ is a positive distribution and hence a Radon measure.  Taking $\xi=e_i$ and $\xi=(e_i+e_j)/\sqrt2$ and using polarization produces all matrix entries; positivity for every $\xi$ gives positive semidefiniteness of the matrix measure.

Choose $\eta\in C_c^\infty(B_{2r}(x))$ with $0\le\eta\le1$, $\eta=1$ on $B_r(x)$, and $|D^2\eta|\le C_nr^{-2}$.  Since $\Delta u$ is a positive measure,
\begin{align*}
\Delta u(B_r(x))
&\le\int\eta\,d(\Delta u)
=\int u\,\Delta\eta\\
&=\int (u-u(x))\Delta\eta
\le \left|\int (u-u(x))\Delta\eta\right|
\le C_n\operatorname{Lip}(u;B_{2r}(x))r^{n-1}.
\end{align*}
Here we used $\int\Delta\eta=0$.  Let $K\Subset U$ and choose $K'\Subset U$ with $K\subset\operatorname{int}K'$.  Convexity gives a finite Lipschitz constant $L$ for $u$ on $K'$.  If $E\subset K$ and $\cH^{n-1}(E)=0$, then for every $\eps>0$ one may cover $E$ by balls $B_{r_i}(x_i)$ with $B_{2r_i}(x_i)\subset K'$ and
$\sum_i r_i^{n-1}<\eps$.  Applying \eqref{eq:Laplacian-growth} and countable subadditivity gives
\[
\Delta u(E)\le C_nL\sum_i r_i^{n-1}\le C_nL\eps,
\]
so $\Delta u(E)=0$.

Finally, if $M=D^2u$ and $A$ is Borel, positivity means that the matrix $M(A)$ is positive semidefinite.  For the Frobenius norm,
$|M(A)|\le\tr M(A)=\Delta u(A)$.  Taking sums over Borel partitions shows that the total variation measure satisfies
$|D^2u|\le\Delta u$ (and hence also the displayed estimate with a dimensional constant for any fixed matrix norm).  This completes the proof.
\end{proof}

\begin{proof}[Proof of \Cref{cor:W21-intro}]
Fix $K\Subset\Omega$.  Convexity makes $u$ Lipschitz on a neighborhood of $K$.  On $\Reg(u)$, the distributional Hessian agrees with the smooth Hessian and is absolutely continuous.  Therefore the singular part of the Hessian measure is supported on $\Sing(u)$.  By \Cref{thm:main},
\[
\cH^{n-1}(\Sing(u)\cap K)=0,
\]
and by \Cref{lem:Hessian-measure-growth} the Hessian measure does not charge this set.  Hence $D^2u$ has no singular part on $K$, and its density belongs to $L^1(K)$.  Since $K$ was arbitrary, $u\in W^{2,1}_{\mathrm{loc}}(\Omega)$.
\end{proof}

\section{Sectional mean values and the geometric origin of the argument}\label{sec:sectional-formulas}

This section is logically independent of the proof of \Cref{thm:main}.  It may be read after \Cref{sec:contact-regularity,sec:john,sec:section-decay,sec:main-proof}.  We first prove two companion Grassmannian formulas and then explain why the first of them nevertheless points toward the contact-set threshold and the finite-scale proof architecture.

\subsection{The null-sectional asymptotic mean value}

Let the eigenvalues of $A\in\Sym(n)$ be ordered as
\[
\lambda_1(A)\le\cdots\le\lambda_n(A),
\]
and define
\[
\cP_q^-(A):=\lambda_1(A)+\cdots+\lambda_q(A).
\]
Ky Fan's minimum principle gives
\begin{equation}\label{eq:Ky-Fan}
\cP_q^-(A)=
\min_{E\in\Gr(q,n)}\tr(P_EA),
\end{equation}
where $P_E$ is orthogonal projection onto $E$.  For completeness, if
$e_1,\ldots,e_q$ is an orthonormal basis of $E$ and
$v_1,\ldots,v_n$ is an eigenbasis of $A$, then
\[
\tr(P_EA)=\sum_{i=1}^q\ip{Ae_i}{e_i}
=\sum_{j=1}^n\lambda_j(A)c_j,
\qquad
c_j:=\sum_{i=1}^q|\ip{e_i}{v_j}|^2,
\]
where $0\le c_j\le1$ and $\sum_jc_j=q$.  The last linear expression is minimized by taking $c_j=1$ for $j\le q$ and $c_j=0$ otherwise, and equality is attained for the span of the first $q$ eigenvectors.  In particular, the minimum value is Lipschitz even when its minimizing planes are not:
\begin{equation}\label{eq:Ky-Fan-Lipschitz}
\bigl|\cP_q^-(A)-\cP_q^-(B)\bigr|
\le q\|A-B\|_{\mathrm{op}}.
\end{equation}
Indeed, apply \eqref{eq:inf-stability-elementary} below to
$E\mapsto\tr(P_EA)$ and $E\mapsto\tr(P_EB)$, and use
$|\tr(P_E(A-B))|\le q\|A-B\|_{\mathrm{op}}$.

\begin{proposition}[Quantitative sectional mean-value expansion]\label{prop:mean-expansion}
Let $1\le q\le n$, let $\phi\in C^2(\Omega)$, and let $K\Subset\Omega$.  For
\[
0<r<d_K:=\dist(K,\partial\Omega),
\]
define
\begin{equation}\label{eq:hessian-modulus}
\omega_{K,\phi}(r):=
\sup\bigl\{\|D^2\phi(y)-D^2\phi(x)\|_{\mathrm{op}}:
 x\in K,\ y\in\Omega,\ |y-x|\le r\bigr\}.
\end{equation}
Then the following assertions hold.

\begin{enumerate}[label=\textup{(\roman*)},leftmargin=2.2em]
\item For every $x\in K$, the map
\[
E\longmapsto \fintavg_{B_r^E(x)}\phi\,d\cH^q
\]
is continuous on $\Gr(q,n)$, and hence its infimum is attained.

\item Whenever $u,v$ are bounded Borel functions on $B_r(x)$,
\begin{equation}\label{eq:mean-nonexpansive}
\left|
\inf_{E\in\Gr(q,n)}\fintavg_{B_r^E(x)}u
-
\inf_{E\in\Gr(q,n)}\fintavg_{B_r^E(x)}v
\right|
\le \|u-v\|_{L^\infty(B_r(x))}.
\end{equation}

\item Uniformly for $x\in K$,
\begin{equation}\label{eq:mean-expansion}
\left|
\inf_{E\in\Gr(q,n)}
\fintavg_{B_r^E(x)}\phi(y)\,d\cH^q(y)
-
\phi(x)-\frac{r^2}{2(q+2)}\cP_q^-(D^2\phi(x))
\right|
\le
\frac{q}{2(q+2)}\omega_{K,\phi}(r)r^2.
\end{equation}
In particular, the left-hand side is $o_K(r^2)$ as $r\downarrow0$.
\end{enumerate}
\end{proposition}

\begin{proof}
Choose an orthonormal frame $U:\R^q\to\R^n$ with image $E$.  Then
\begin{equation}\label{eq:frame-representation}
\fintavg_{B_r^E(x)}\phi\,d\cH^q
=
\fintavg_{B_r^q(0)}\phi(x+Uz)\,dz.
\end{equation}
The right-hand side depends continuously on $U$.  It is invariant under replacing $U$ by $UO$, $O\in O(q)$, and therefore descends to a continuous function on the Grassmannian.  Since $\Gr(q,n)$ is compact, the minimum is attained.  This proves (i).

For arbitrary families of real numbers $\{a_E\}$ and $\{b_E\}$ indexed by the same set,
\begin{equation}\label{eq:inf-stability-elementary}
\left|\inf_E a_E-\inf_E b_E\right|
\le \sup_E|a_E-b_E|.
\end{equation}
Applying this to the sectional averages of $u$ and $v$ proves (ii).  Notice that this estimate does not require a continuous choice of a minimizing plane.

For $x\in K$ and $|z|<r$, the integral form of Taylor's theorem gives
\begin{align}
\phi(x+z)
={}&\phi(x)+D\phi(x)\cdot z
+\frac12\ip{D^2\phi(x)z}{z}+R_x(z),
\label{eq:taylor-integral}\\
R_x(z)
={}&\int_0^1(1-t)
\ip{\bigl(D^2\phi(x+tz)-D^2\phi(x)\bigr)z}{z}\,dt.
\label{eq:taylor-remainder}
\end{align}
Thus
\begin{equation}\label{eq:uniform-taylor-remainder}
|R_x(z)|\le \frac12\omega_{K,\phi}(r)|z|^2
\end{equation}
for every $x\in K$, every $E\in\Gr(q,n)$, and every $z\in E$ with $|z|<r$.  The point is that the modulus in \eqref{eq:hessian-modulus} is taken in the ambient space and is independent of $E$.

The linear term has zero sectional average.  In an orthonormal basis of $E$,
\[
\fintavg_{B_r^E(0)}z_i z_j\,d\cH^q(z)
=\frac{r^2}{q+2}\delta_{ij},
\qquad
\fintavg_{B_r^E(0)}|z|^2\,d\cH^q(z)
=\frac{q}{q+2}r^2.
\]
Consequently, if
\[
A_E(x,r):=\fintavg_{B_r^E(x)}\phi\,d\cH^q,
\]
then
\begin{equation}\label{eq:uniform-plane-expansion}
A_E(x,r)
=
\phi(x)+\frac{r^2}{2(q+2)}\tr(P_ED^2\phi(x))+\mathcal R_E(x,r),
\end{equation}
where
\begin{equation}\label{eq:uniform-plane-error}
\sup_{x\in K,\,E\in\Gr(q,n)}|\mathcal R_E(x,r)|
\le
\frac{q}{2(q+2)}\omega_{K,\phi}(r)r^2.
\end{equation}
Applying \eqref{eq:inf-stability-elementary} to \eqref{eq:uniform-plane-expansion} and using Ky Fan's identity \eqref{eq:Ky-Fan} proves \eqref{eq:mean-expansion}.

Finally, $D^2\phi$ is uniformly continuous on a fixed compact neighborhood of $K$, so $\omega_{K,\phi}(r)\to0$.  No continuity of the minimizing eigenspace is used.  Indeed, even if the set of minimizers jumps at an eigenvalue crossing, both the minimum in \eqref{eq:Ky-Fan} and the error estimate \eqref{eq:uniform-plane-error} remain stable.
\end{proof}

\begin{remark}[Jumping minimizing planes]\label{rem:jumping-minimizers}
The map $u\mapsto\cM_{q,r}u(x)$ is nonlinear, but \eqref{eq:mean-nonexpansive} shows that it is $1$-Lipschitz for the uniform norm.  Degeneracy of $D^2\phi(x)$ may cause the set of minimizing planes in \eqref{eq:Ky-Fan} to be nonunique and may make a chosen minimizer discontinuous under perturbation.  This affects only the choice of an argmin.  The value of the infimum is protected by \eqref{eq:inf-stability-elementary}, while the Taylor error is bounded before optimization by \eqref{eq:uniform-plane-error}.  Hence no spectral gap and no continuous selection of minimizing eigenspaces is required.
\end{remark}

\begin{corollary}[Homogeneous convex $k$-Hessian equation]\label{cor:mean-homogeneous}
Let $q=n-k+1$ and let $\phi\in C^2$ be convex.  At every point $x$,
\begin{equation}\label{eq:sigmak-Pq-equivalence}
\sigma_k(D^2\phi(x))=0
\quad\Longleftrightarrow\quad
\cP_q^-(D^2\phi(x))=0.
\end{equation}
Consequently, at such a point,
\begin{equation}\label{eq:asymptotic-sectional-MVP}
\phi(x)=
\inf_{E\in\Gr(q,n)}
\fintavg_{B_r^E(x)}\phi\,d\cH^q+o(r^2).
\end{equation}
\end{corollary}

\begin{proof}
For a positive semidefinite matrix, $\sigma_k=0$ if and only if the rank is at most $k-1$, which is equivalent to having at least $n-k+1=q$ zero eigenvalues.  Apply \Cref{prop:mean-expansion}.
\end{proof}

\begin{remark}\label{rem:mean-not-regularizing}
The formula is an infinitesimal characterization of flat directions, not a regularization theorem for the homogeneous equation.  If $F\subset\R^n$ has dimension at most $k-1$ and $g:F\to\R$ is convex, then $u(x)=g(P_Fx)$ is affine on every affine $q$-plane parallel to a subspace of $F^\perp$.  Such a function can be singular on a large set.
\end{remark}

\subsection{A complementary Grassmannian identity}

The Grassmannian is the compact homogeneous space
$O(n)/(O(k)\times O(n-k))$ and therefore carries a unique normalized
$O(n)$-invariant probability measure, denoted by $\gamma_{k,n}$.  If
$V\in\Gr(k,n)$ and $U:\R^k\to\R^n$ is an orthonormal frame for $V$, set
\[
A|_V:=U^TAU.
\]

\begin{proposition}[Grassmannian average of sectional determinants]\label{prop:Grassmann}
For every $A\in\Sym(n)$,
\begin{equation}\label{eq:Grassmann-identity}
\sigma_k(A)=\binom nk
\int_{\Gr(k,n)}\det(A|_V)\,d\gamma_{k,n}(V).
\end{equation}
\end{proposition}

\begin{proof}
By orthogonal invariance, assume that $A=\operatorname{diag}(\lambda_1,\ldots,\lambda_n)$.  If $U$ is an $n\times k$ matrix with orthonormal columns spanning $V$, Cauchy--Binet gives
\[
\det(U^TAU)=
\sum_{|I|=k}\lambda_I\det(U_I)^2,
\qquad
\lambda_I:=\prod_{i\in I}\lambda_i.
\]
Haar symmetry makes the averages of $\det(U_I)^2$ independent of $I$.  A second application of Cauchy--Binet to $U^TU=I_k$ gives
\[
\sum_{|I|=k}\det(U_I)^2=1.
\]
Thus each of the $\binom nk$ equal averages is $\binom nk^{-1}$, and integration yields \eqref{eq:Grassmann-identity}.
\end{proof}

For convex $u$, \eqref{eq:Grassmann-identity} identifies the $k$-Hessian density with the Grassmannian average of the Monge--Amp\`ere densities of the restrictions $u|_{x+V}$.

\begin{corollary}[A positive proportion of nondegenerate sections]\label{cor:many-good-planes}
Let $A\ge0$, let $N=\binom nk$, and suppose $\sigma_k(A)>0$.  Then
\begin{equation}\label{eq:good-plane-measure}
\gamma_{k,n}\left(
\left\{V\in\Gr(k,n):
\det(A|_V)\ge\frac{\sigma_k(A)}{2N}
\right\}
\right)
\ge\frac{1}{2N-1}.
\end{equation}
\end{corollary}

\begin{proof}
By the Cauchy--Binet representation in the proof of \Cref{prop:Grassmann}, $0\le\det(A|_V)\le\sigma_k(A)$.  If the set in \eqref{eq:good-plane-measure} has measure $\theta$, then \eqref{eq:Grassmann-identity} implies
\[
\frac{\sigma_k(A)}{N}
\le
(1-\theta)\frac{\sigma_k(A)}{2N}
+\theta\sigma_k(A).
\]
Solving for $\theta$ gives the claim.
\end{proof}

The identity and \Cref{cor:many-good-planes} are not used in the proof of \Cref{thm:main}; they clarify how the positive equation and the homogeneous flat-direction equation are seen on two complementary Grassmannians.

\subsection{From Newton-tensor flux to supporting contact geometry}\label{sec:mean-to-contact}

We now explain the route from the original mean-value question to the proof above.  The point is not that the sectional formula supplies a missing estimate in \Cref{sec:contact-regularity,sec:john,sec:section-decay}; it does not.  Rather, it singles out the correct degeneracy and suggests its nonsmooth finite-scale replacement.

For a smooth function $u$ in Euclidean space, let
\[
T_{k-1}^{ij}(D^2u):=
\frac{\partial\sigma_k}{\partial u_{ij}}(D^2u)
\]
be the $(k-1)$st Newton tensor.  Symmetry of third derivatives gives the standard identities
\begin{equation}\label{eq:newton-flux}
\partial_iT_{k-1}^{ij}(D^2u)=0,
\qquad
T_{k-1}^{ij}(D^2u)u_{ij}=k\sigma_k(D^2u),
\end{equation}
and hence
\begin{equation}\label{eq:newton-divergence}
k\sigma_k(D^2u)
=
\partial_i\bigl(T_{k-1}^{ij}(D^2u)u_j\bigr).
\end{equation}
For $k=1$, $T_0=I$, so \eqref{eq:newton-divergence} is the ordinary Laplace flux identity whose integration on balls differentiates the spherical mean.  For $k=2$, its flat form is the Newton-tensor identity underlying Reilly's Hessian formula; on a curved manifold, commutation of covariant derivatives produces curvature terms \cite{Reilly1977}.  For $k\ge2$, however, the flux weight $T_{k-1}(D^2u)$ depends on the unknown Hessian.  Thus \eqref{eq:newton-divergence} does not differentiate any fixed linear average of $u$ over Euclidean spheres.

This failure suggests replacing a fixed average by an adaptive one.  On the convex branch,
\[
\sigma_k(D^2u)=0
\quad\Longleftrightarrow\quad
\rank D^2u\le k-1
\quad\Longleftrightarrow\quad
\dim\ker D^2u\ge q=n-k+1.
\]
The second-order term of the average over a $q$-plane $E$ is
\[
\frac{r^2}{2(q+2)}\tr(P_ED^2u),
\]
and minimizing over $E$ produces $\cP_q^-(D^2u)$.  This is exactly the content of \Cref{prop:mean-expansion,cor:mean-homogeneous}.  In this sense, the null-sectional mean-value operator is the adaptive average naturally selected by the homogeneous convex equation.

Two obstructions prevent this observation from proving partial regularity directly.  First, it is a $C^2$ infinitesimal statement, whereas the points under investigation need not possess a Hessian.  Second, even for smooth convex solutions of $\sigma_k(D^2u)=1$, the equation gives no universal positive lower bound for $\cP_q^-(D^2u)$: the eigenvalues may be arbitrarily anisotropic while their $k$th elementary symmetric function remains fixed.  The homogeneous mean-value property itself is not regularizing, as \Cref{rem:mean-not-regularizing} shows.

The finite-scale replacement of a null eigenspace is a supporting contact set.  After subtracting a supporting affine function, an $m$-dimensional contact set records exact affine flatness in $m$ directions without assuming differentiability.  The threshold predicted by the mean-value calculation is precisely $m=q$:
\begin{itemize}[leftmargin=2.2em]
\item if one supporting contact set has dimension at most $q-1=n-k$, its orthogonal complement contains $k$ directions; the admissible saddle of \Cref{lem:saddle} and the Chou--Wang estimate then give regularity;
\item hence every supporting contact set at a singular point has at least $q$ directions;
\item the positive right-hand side, tested by a John ellipsoidal paraboloid, controls the product of the remaining smallest $k$ axes as in \eqref{eq:axis-product};
\item the $q$ contact directions keep the large axes from collapsing, which yields the section-volume deficit needed by Mooney's covering theorem.
\end{itemize}
Thus the conceptual passage is
\begin{equation}\label{eq:conceptual-passage}
\begin{gathered}
\text{null-sectional mean value}
\ \rightsquigarrow\
q\text{-dimensional supporting flatness}\\
\rightsquigarrow\
\text{John-axis volume decay}
\ \rightsquigarrow\
\text{Hausdorff control}.
\end{gathered}
\end{equation}
The first arrow is heuristic at a singular point, but the middle two arrows are the rigorous finite-scale mechanism of \Cref{sec:contact-regularity,sec:john,sec:section-decay}.  This also explains why the sectional formulas are best placed after the proof: they reveal the origin and possible future use of the geometry, while the theorem itself rests on supporting contacts, barriers, and convex sections.

The same viewpoint may become operational in a future blow-up or quantitative theory.  Height-normalized singular blow-ups can drive the right-hand side toward zero, so the optimized $q$-planes in \Cref{prop:mean-expansion} are natural candidates for approximate flat directions.  Turning that observation into packing estimates requires coherence of these planes across points and scales, precisely the missing issue described in \Cref{sec:quantitative-limitations}.

\section{Sharpness}\label{sec:sharpness}

The exponent in \Cref{thm:main} is optimal in the intermediate range.  We first record the elementary extension mechanism.

\begin{proposition}[Cylindrical extension]\label{prop:cylindrical-extension}
Let $1\le k\le n$, let $U\subset\R^k$ be open and convex, and let $w$ be a convex viscosity solution of
\[
\det D^2w=1\quad\text{in }U,
\]
and define
\[
u(x,z):=w(x),
\qquad
(x,z)\in U\times\R^{n-k}.
\]
Then $u$ is a convex viscosity solution of
\[
\sigma_k(D^2u)=1
\quad\text{in }U\times\R^{n-k}.
\]
\end{proposition}

\begin{proof}
It suffices to argue locally.  On a ball compactly contained in $U$, approximate $w$ uniformly by smooth convex solutions $w_j$ of $\det D^2w_j=1$ using \Cref{thm:CNS} in dimension $k$.  Set $u_j(x,z)=w_j(x)$.  Then
\[
D^2u_j=
\begin{pmatrix}
D^2w_j&0\\
0&0
\end{pmatrix}
\in\Gamma_k,
\qquad
\sigma_k(D^2u_j)=\det D^2w_j=1.
\]
The functions $u_j$ converge locally uniformly to $u$, and viscosity stability gives the conclusion.
\end{proof}

Mooney constructed, in every dimension $k\ge3$, convex Alexandrov solutions of $\det D^2w=1$ containing compact subsets of the non-strictly-convex set with Hausdorff dimension arbitrarily close to $k-1$ \cite[Section~4, especially Remark~4.2]{Mooney2015}.  For continuous positive densities, convex Alexandrov and viscosity solutions of the Monge--Amp\`ere equation coincide; see \cite[Chapter~1]{Gutierrez2016}.  Thus these examples belong to the viscosity class used here.  Moreover, every point of the non-strictly-convex set is in $\Sing(w)$: if $w$ were $C^2$ in a neighborhood of such a point, then $D^2w\ge0$ and $\det D^2w=1$ would make $D^2w$ positive definite there; by continuity it would remain uniformly positive definite on a smaller ball, forcing strict convexity.

If $w$ is not $C^2$ near $x$, then the cylindrical extension cannot be $C^2$ near any $(x,z)$, because restriction to the affine slice $\R^k\times\{z\}$ would make $w$ locally $C^2$.  Hence
\[
\Sing(w)\times\R^{n-k}
\subset\Sing(u).
\]
For completeness, the required product estimate follows directly from Frostman's lemma.  If $s<\dim_{\cH}K$, choose a finite measure $\mu$ supported on $K$ with $\mu(B_r)\le C_sr^s$.  Lebesgue measure $\nu$ on $[0,1]^{n-k}$ satisfies $\nu(B_r)\le C_nr^{n-k}$.  Therefore the product measure obeys
\[
(\mu\times\nu)(B_r(x,z))\le C r^{s+n-k},
\]
so another application of Frostman's lemma yields
\[
\dim_{\cH}\bigl(K\times[0,1]^{n-k}\bigr)\ge s+n-k.
\]
Letting $s\uparrow\dim_{\cH}K$ gives the standard product lower bound; see also \cite[Chapter~8]{Mattila1995}.  Consequently, for every $\delta>0$ and every $3\le k<n$, there are convex viscosity solutions of $\sigma_k(D^2u)=1$ whose singular sets have Hausdorff dimension at least $n-1-\delta$.  Thus no estimate of the form
\[
\dim_{\cH}\Sing(u)\le n-1-\delta(n,k)
\]
can hold uniformly.

\section{Further directions}\label{sec:further}

\subsection{Variable right-hand sides}\label{sec:variable-rhs}

The geometric portion of the argument is insensitive to roughness of the right-hand side: \Cref{thm:stratification,lem:axis-product,prop:section-decay} use only the viscosity lower bound
\[
\sigma_k(D^2u)\ge\lambda>0.
\]
For an equation
\begin{equation}\label{eq:variable-rhs}
\sigma_k(D^2u)=f(x),
\qquad
0<\lambda\le f\le\Lambda,
\end{equation}
the obstruction enters only when one tries to prove the support-dependent regularity criterion, and it enters at two logically separate stages.

\smallskip
\noindent\emph{Approximation and stability.}
For constant right-hand side, \Cref{lem:smooth-approximation} uses smooth Dirichlet solutions of the same equation, direct comparison with the limiting viscosity solution, and a gradient estimate whose constants are uniform along the approximation.  If $f$ is continuous, one can choose smooth positive $f_j\to f$ uniformly, but the approximating equations now vary with $j$.  The constant-right-hand-side Dirichlet theorem quoted in \Cref{thm:CNS} does not by itself furnish the required classical solutions for the pairs $(f_j,g_j)$: one must separately verify the admissible subsolution and boundary-barrier hypotheses in the variable-data and degenerate Dirichlet theory, with estimates uniform in $j$; see, for example, \cite{IvochkinaTrudingerWang2005}.  Even when these classical solutions exist, viscosity stability and uniqueness identify a locally uniform subsequential limit only after one has obtained uniform compactness estimates for them.  If $f$ is merely discontinuous or belongs only to $L^p$, no sequence of smooth functions converges to $f$ uniformly in general.  The standard local-uniform stability theorem for viscosity solutions then sees only upper and lower relaxed limits of the right-hand sides and does not, by itself, select a prescribed solution of \eqref{eq:variable-rhs}; see, for example, \cite{CrandallIshiiLions1992}.  One would need an $L^p$-viscosity or Hessian-measure formulation together with a matching comparison and approximation theorem.  Thus the direct analogue of \Cref{lem:smooth-approximation} is already nonautomatic at the discontinuous level.

\smallskip
\noindent\emph{Loss of uniform a priori estimates.}
The decisive obstruction is the variable-right-hand-side Pogorelov estimate.  Chou--Wang's estimate assumes $C^{1,1}$ control of the right-hand side and its constant depends on that norm and on a positive lower bound; see \cite[Theorem~4.1]{ChouWang2001} and \cite[Theorem~4.3]{Wang2009}.  Their proof differentiates the equation twice.  Written for $F=\sigma_k^{1/k}$, it contains a term involving $D^2(f^{1/k})$.  If $f_\eps$ is a standard mollification of a positive $f\in C^{0,\alpha}$, then
\begin{equation}\label{eq:rhs-mollification-loss}
\|Df_\eps\|_{L^\infty}
\lesssim \eps^{\alpha-1}[f]_{C^{0,\alpha}},
\qquad
\|D^2f_\eps\|_{L^\infty}
\lesssim \eps^{\alpha-2}[f]_{C^{0,\alpha}}.
\end{equation}
Because $f_\eps\ge\lambda/2$ after a harmless adjustment, the chain rule gives the same leading loss for $D^2(f_\eps^{1/k})$.  Thus, for Lipschitz $f$, the first derivatives can remain uniformly bounded but the second derivatives typically grow like $\eps^{-1}$; for $0<\alpha<1$, even the first-derivative bound generally diverges.  The Chou--Wang constants therefore blow up along the smooth approximation, so the uniform Hessian bound \eqref{eq:uniform-Hessian-bound} cannot be passed to the limit.  Convexity of the limiting solution does not repair this step, since the Caffarelli--Nirenberg--Spruck approximants for $k<n$ are $k$-admissible but need not be convex.

The fixed saddle quadratic is not the source of the failure.  For a smooth positive $f_j$, concavity and homogeneity still imply $L_{u_j}w>0$ for the same barrier.  What is missing is a gradient and Hessian estimate uniform in $j$.  A positive $C^{1,1}$ right-hand side is compatible with the known Chou--Wang estimate if the approximation preserves a uniform $C^{1,1}$ bound, and the present method should extend to that regime after a separate variable-right-hand-side approximation lemma is supplied.  We do not claim that extension here.  For Hölder, merely continuous, or $L^p$ right-hand sides, one needs either a Pogorelov estimate that avoids differentiating $f$ twice or a different regularity mechanism.  The high-contact portion of the singular set, however, is already controlled by \Cref{thm:stratification} under the sole lower bound in \eqref{eq:variable-rhs}.

\subsection{Quantitative stratification}\label{sec:quantitative-limitations}

The mean-value side is quantitatively stable: \eqref{eq:mean-expansion} gives an explicit modulus-of-continuity remainder, and \eqref{eq:section-decay-quantitative} gives a point-independent section-volume estimate once a contact inradius $\rho$ is prescribed.  The unresolved issue is to produce such a normalized inradius from a finite-scale failure of regularity.  The qualitative condition $d_u(x,p)\ge q$ says only that the exact limiting contact set contains some nondegenerate $q$-simplex.  Its inradius may tend to zero arbitrarily rapidly as $(x,p)$ varies, and exact contact may disappear under perturbation, leaving only an approximate contact set.  In addition, \eqref{eq:axis-product} controls only the product of the smallest $k$ John axes.  It gives no uniform lower bound for any prescribed individual axis and no coherence of the associated principal planes between neighboring points or consecutive scales.

A quantitative theory would require a normalized dichotomy with universal parameters.  Schematically, on each ball and at each scale one would need either an effective Hessian bound on a smaller concentric ball, or a $q$-dimensional approximate contact simplex with a fixed normalized inradius and a controlled height error.  To convert this into a packing or tubular-neighborhood estimate, one would also need a cone-splitting or Reifenberg-type statement forcing the approximate flat directions to be coherent across nearby bad points and adjacent scales, in the spirit of quantitative stratification methods such as \cite{CheegerNaber2013}.  Neither assertion follows from the present exact-contact argument.  Thus \eqref{eq:section-decay-quantitative} is a useful weak quantitative substitute, but the current paper deliberately claims Hausdorff-nullity rather than a Minkowski-content bound.

\subsection{Higher integrability beyond \texorpdfstring{$W^{2,1}$}{W2,1}}\label{sec:higher-integrability}

\Cref{cor:W21-intro} removes the singular part of the distributional Hessian, but it is an endpoint qualitative statement: the proof does not bound the distribution function of the absolutely continuous density.  This is substantially weaker than the Monge--Amp\`ere higher-integrability theory.  In that setting, De Philippis--Figalli obtain $L\log^mL$ control for every finite $m$, and De Philippis--Figalli--Savin and Schmidt obtain an $L^{1+\varepsilon}$ gain \cite{DePhilippisFigalli2013,DePhilippisFigalliSavin2013,Schmidt2013}.  Their arguments use normalized sections and the volume-preserving affine covariance of the determinant, together with the localization geometry initiated by Caffarelli \cite{Caffarelli1990Localization,Caffarelli1990W2p}.

For $k<n$, a general determinant-one affine change does not preserve $\sigma_k(D^2u)$.  The present proof retains only an anisotropic product estimate for the smallest $k$ John axes, namely \eqref{eq:axis-product}; it gives neither a canonical normalization of all sections nor a scale-by-scale distributional estimate for $|D^2u|$.  A plausible route to $L\log L$, weak-$L^{1+\varepsilon}$, or $W^{2,1+\varepsilon}$ bounds would therefore require genuinely quantitative versions of both parts of the proof: a finite-scale regularity/contact dichotomy and a packing estimate that is coherent across neighboring points and adjacent scales.  The quantitative stratification problem described above is thus not only a refinement of the Hausdorff theorem, but also the natural obstruction to stronger Sobolev integrability.

\subsection{A measure-level sectional formula}

For smooth convex functions, \eqref{eq:Grassmann-identity} can be integrated and disintegrated over affine $k$-planes.  A measure-level version for Hessian measures should relate the $k$-Hessian measure to averaged Monge--Amp\`ere measures of restrictions.  The weak continuity theory of Hessian measures in \cite{TrudingerWang1999} provides a natural approximation framework.  Such a formula is not needed for the present proof, but it may be useful for quantitative slicing arguments.

\subsection{Removable and isolated singularities}

The singular set studied here should be distinguished from a removable
singularity set: throughout the paper, the viscosity solution is already
defined on all of $\Omega$, and the issue is failure of local $C^2$
regularity.  For punctured-domain problems, Labutin obtained sharp
removability criteria and a classification of isolated singularities for
fully nonlinear uniformly elliptic equations
\cite{Labutin2000,Labutin2001}; his nonlinear potential theory for
$k$-subharmonic functions gives capacity and Hausdorff-measure criteria that
are specific to the $k$-Hessian operator \cite{Labutin2002}.

For the homogeneous $k$-Hessian equation, exterior and punctured Dirichlet
problems with fundamental asymptotics have been studied by Ma--Zhang and
Gao--Ma--Zhang \cite{MaZhang2022,GaoMaZhang2023}.  In the Monge--Amp\`ere
setting, Jian--Tu's strong maximum principle identifies strict convexity as
the natural locus for comparison of generalized solutions
\cite{JianTu2025}, while recent sharp global Alexandrov estimates of
Jin--Tu--Xiong yield rigidity and applications to entire solutions with
isolated singularities \cite{JinTuXiong2026}.  It would be interesting to
understand whether the supporting-contact stratification developed here can
interact with these punctured-domain and capacity theories, especially for
the homogeneous blow-up equation $\sigma_k(D^2u)=0$.

\clearpage
\appendix

\section{The higher-codimension section-covering criterion}
\label{app:higher-codimension-covering}

The case $s=1$ of the criterion below is Mooney's section-covering theorem
\cite[Theorem~3.1]{Mooney2015}.  The higher-codimension statement is recorded
in \cite[Remark~3.4]{Mooney2015} without the details of the argument.  We give
a complete proof, following the same two mechanisms as in the codimension-one
case: degeneration of John axes and concentration of Monge--Amp\`ere measure.
For a convex function $v$, we write
\[
  M_v(A):=|\partial v(A)|
\]
for its Alexandrov Monge--Amp\`ere measure.

\begin{theorem}[Higher-codimension section covering]
\label{thm:appendix-covering}
Let $v$ be convex in $B_1\subset\R^n$, let $1\le s\le n-1$ be an integer,
and let $E\subset B_1$.  Assume that for every $x\in E$ and every
$p\in\partial v(x)$ there exist constants $C_{x,p}<\infty$ and
$h_{x,p}>0$ such that
\begin{equation}\label{eq:appendix-section-hypothesis}
  |S^v_{h,p}(x)|\le C_{x,p}h^{(n+s)/2}
  \qquad (0<h<h_{x,p}).
\end{equation}
Then
\[
  \cH^{n-s}(E)=0.
\]
\end{theorem}

We first record two elementary convex facts used in the proof.

\begin{lemma}[Sublevel normals and convex-hull volume]
\label{lem:appendix-convex-facts}
Let $f$ be finite and convex on a neighborhood of the compact convex body
$K=\{f\le h\}$.  If $y\in\partial K$ and $\nu$ is an outward unit normal
to $K$ at $y$, then
\begin{equation}\label{eq:normal-subgradient}
  a\nu\in\partial f(y)
\end{equation}
for some $a>0$.

Moreover, if $\mathcal E=z+AB_1$ is an ellipsoid and
$|A^{-1}(y-z)|\ge2$, then
\begin{equation}\label{eq:convex-hull-volume}
 |\operatorname{co}(\mathcal E\cup\{y\})|
 \ge c_n |A^{-1}(y-z)|\,|\mathcal E|.
\end{equation}
\end{lemma}

\begin{proof}
The first assertion is the normal-cone formula for a convex sublevel set,
\[
 N_K(y)=\{tq:t\ge0,\ q\in\partial f(y)\}.
\]
It applies because $K$ has a strict interior point.  For completeness, this
formula follows from the convex multiplier rule applied to the minimization
of $-\nu\cdot x$ subject to $f(x)\le h$; the multiplier is positive because
$\nu\ne0$.

For the second assertion, apply the affine map $x\mapsto A^{-1}(x-z)$.  We
are reduced to the convex hull of $B_1$ and a point $Y$ with $|Y|\ge2$.
That hull contains a cone of height at least $|Y|/2$ whose base is an
$(n-1)$-ball of dimensional radius.  Its volume is at least
$c_n|Y||B_1|$.  Transforming back proves \eqref{eq:convex-hull-volume}.
\end{proof}

\begin{lemma}[Alexandrov maximum estimate]
\label{lem:appendix-alexandrov-maximum}
Let $D\subset\mathbb R^d$ be a bounded convex domain and let
$f\in C(\overline D)$ be convex, with $f=0$ on $\partial D$.  Then
\begin{equation}\label{eq:appendix-alexandrov-maximum}
  M_f(D)\,|D|\ge c_d\,\bigl|\min_D f\bigr|^d.
\end{equation}
\end{lemma}

This is the standard affine-invariant Alexandrov maximum estimate; see
\cite[Lemma~2.5]{Mooney2015}.  We state it here to make explicit the only
previous section-geometry estimate used below.

\begin{lemma}[Restriction and lifting of Monge--Amp\`ere mass]
\label{lem:appendix-mass-lifting}
Assume that
\begin{equation}\label{eq:appendix-strong-convexity}
 v(z)\ge v(y)+p\cdot(z-y)+\frac12|z-y|^2
 \qquad(p\in\partial v(y)).
\end{equation}
Let $V\subset\R^n$ be a linear subspace of dimension $d$, let
$w=v|_V$, and let $D\subset V\cap B_{r/4}$.  Then
\begin{equation}\label{eq:mass-lifting}
  M_v(B_r)\ge c_n r^{n-d}M_w(D).
\end{equation}
\end{lemma}

\begin{proof}
For $y\in V$, the standard subdifferential rule for restriction gives
\begin{equation}\label{eq:restriction-subgradient}
  \partial_Vw(y)=P_V\partial v(y).
\end{equation}
Indeed, one inclusion follows by projection.  Conversely, if
$q\in\partial_Vw(y)$, then $y$ minimizes $v(z)-q\cdot P_Vz$ under the affine
constraint $z\in V$.  The convex optimality condition for an affine
constraint gives a vector $\zeta\in V^\perp$ such that
$q+\zeta\in\partial v(y)$, proving the reverse inclusion.

Fix $y\in B_{r/4}$ and $p\in\partial v(y)$.  We claim that
\begin{equation}\label{eq:subgradient-ball}
  B_{r/8}(p)\subset\partial v(B_r).
\end{equation}
If $|q-p|<r/8$, minimize $v(z)-q\cdot z$ on
$\overline B_{r/2}(y)$.  On its boundary, \eqref{eq:appendix-strong-convexity}
gives
\[
 [v(z)-q\cdot z]-[v(y)-q\cdot y]
 \ge \frac12|z-y|^2-|q-p|\,|z-y|>0.
\]
The minimum is therefore attained in the interior of $B_{r/2}(y)\subset
B_r$, which proves \eqref{eq:subgradient-ball}.

Let $Q=\partial_Vw(D)$.  By \eqref{eq:restriction-subgradient}, for every
$q_0\in Q$ there are $y\in D$ and $p\in\partial v(y)$ with $P_Vp=q_0$.
Equation \eqref{eq:subgradient-ball} shows that the fiber of
$\partial v(B_r)$ over $q_0$ contains an $(n-d)$-dimensional ball of radius
$r/8$ in $V^\perp$.  Fubini's theorem in the orthogonal decomposition
$\R^n=V\oplus V^\perp$ yields
\[
 |\partial v(B_r)|\ge c_n r^{n-d}|Q|,
\]
which is \eqref{eq:mass-lifting}.
\end{proof}

\begin{proof}[Proof of \Cref{thm:appendix-covering}]
Set
\[
  \widetilde v(y):=v(y)+\frac12|y|^2.
\]
For $\widetilde p=p+x$ one has
\[
 \widetilde v(y)-\widetilde v(x)-\widetilde p\cdot(y-x)
 =v(y)-v(x)-p\cdot(y-x)+\frac12|y-x|^2.
\]
Thus
$S^{\widetilde v}_{h,\widetilde p}(x)\subset S^v_{h,p}(x)$, and it is
enough to prove the theorem for a function of the form
\begin{equation}\label{eq:strong-convex-form}
  v=v_0+\frac12|x|^2,
  \qquad v_0\ \text{convex}.
\end{equation}
Such a function satisfies \eqref{eq:appendix-strong-convexity}.

\smallskip
\noindent\emph{Step 1: at least $s+1$ John axes degenerate faster than
$\sqrt h$.}
Fix $x\in E$ and $p\in\partial v(x)$.  Translate $x$ to the origin and
subtract the supporting affine function, so that
\[
 x=0,\qquad p=0,\qquad v(0)=0,\qquad v\ge\frac12|x|^2.
\]
For small $h>0$, let $K_h=\overline{S^v_{h,0}(0)}$ and choose a John
ellipsoid
\[
  \mathcal E_h=z_h+A_hB_1,
  \qquad
  \mathcal E_h\subset K_h\subset z_h+C_nA_hB_1.
\]
Write its semiaxes as
\[
  d_1(h)\ge\cdots\ge d_n(h)>0.
\]
We claim that
\begin{equation}\label{eq:axis-degeneration}
  \frac{d_{n-s}(h)}{\sqrt h}\longrightarrow0
  \qquad(h\downarrow0).
\end{equation}

Local Lipschitz continuity gives $B_{ch}(0)\subset S^v_{h,0}(0)$ for all
small $h$, with $c=c(x,p)>0$.  Comparing widths in the upper John inclusion,
\begin{equation}\label{eq:smallest-axis-lower}
  d_n(h)\ge ch.
\end{equation}
The lower John inclusion and \eqref{eq:appendix-section-hypothesis} give
\begin{equation}\label{eq:john-product-upper}
  \prod_{i=1}^n d_i(h)\le C h^{(n+s)/2}.
\end{equation}

Suppose that \eqref{eq:axis-degeneration} fails.  Then for some $\delta>0$
and a sequence $h_j\downarrow0$,
\begin{equation}\label{eq:large-n-s-axis}
  d_{n-s}(h_j)\ge\delta\sqrt{h_j}.
\end{equation}
Equations \eqref{eq:smallest-axis-lower}--\eqref{eq:large-n-s-axis} imply
\begin{equation}\label{eq:all-short-axes-linear}
  ch_j\le d_i(h_j)\le Ch_j
  \qquad(n-s+1\le i\le n).
\end{equation}
Indeed,
\[
 d_{n-s+1}(h_j)(ch_j)^{s-1}
 (\delta\sqrt{h_j})^{n-s}
 \le \prod_{i=1}^n d_i(h_j)
 \le C h_j^{(n+s)/2}.
\]
The same bounds give
\begin{equation}\label{eq:ellipsoid-critical-lower}
  |\mathcal E_{h_j}|\ge c h_j^{(n+s)/2}.
\end{equation}

Let $e_j$ be a shortest principal direction of $\mathcal E_{h_j}$.  The
upper John inclusion and \eqref{eq:all-short-axes-linear} show that $K_{h_j}$
has width $O(h_j)$ in direction $e_j$.  At the two supporting points of
$K_{h_j}$ with outward normals $e_j$ and $-e_j$, respectively,
\Cref{lem:appendix-convex-facts} gives subgradients
$a_j^+e_j$ and $-a_j^-e_j$, where $a_j^\pm>0$.  Since the points lie on
$\{v=h_j\}$ and $0\in K_{h_j}$, the subgradient inequality and the width
bound imply $a_j^\pm\ge c>0$.  Local Lipschitz continuity bounds these
subgradients from above.  Since $K_{h_j}\subset B_{\sqrt{2h_j}}$, the
supporting points tend to the origin.  Passing to a subsequence and using the
closedness of the subdifferential graph,
\[
 e_j\to e,
 \qquad
 a^+e,-a^-e\in\partial v(0)
\]
for some $a^\pm>0$.  Consequently, for all sufficiently small $h$,
\begin{equation}\label{eq:fixed-strip}
  S^v_{h,0}(0)\subset\{|e\cdot x|\le Ch\}.
\end{equation}

Let
\[
  a:=\max\{t\in\R:te\in\partial v(0)\}.
\]
The maximum is finite and attained because the local subdifferential is
compact.  Choose $\Lambda>1+C|a|$, with $C$ as in
\eqref{eq:fixed-strip}.  Then
\begin{equation}\label{eq:section-engulfing}
  S^v_{h_j,0}(0)\subset S^v_{\Lambda h_j,ae}(0).
\end{equation}
We claim that
\begin{equation}\label{eq:long-point}
  R_j:=h_j^{-1}\sup\{e\cdot x:x\in S^v_{\Lambda h_j,ae}(0)\}
  \longrightarrow\infty.
\end{equation}
If $R_j\le R$ along a subsequence, set $b_j=(R+1)h_j$.  No point of the
hyperplane $\{e\cdot x=b_j\}$ belongs to the larger section, so
\[
  v(x)\ge a b_j+\Lambda h_j
  =\left(a+\frac{\Lambda}{R+1}\right)b_j
\]
there.  Applying this inequality at $b_jx/(e\cdot x)$ and using convexity
shows, after $j\to\infty$, that
\[
 v(x)\ge\left(a+\frac{\Lambda}{R+1}\right)e\cdot x
 \qquad(e\cdot x>0).
\]
For $e\cdot x\le0$ the same inequality follows from the supporting plane with
slope $ae$.  This contradicts the maximality of $a$ and proves
\eqref{eq:long-point}.

Choose $x_j$ in the closure of the larger section with
$e\cdot x_j\ge\frac12R_jh_j$.  By convexity and
\eqref{eq:section-engulfing},
\[
 \operatorname{co}(\mathcal E_{h_j}\cup\{x_j\})
 \subset\overline{S^v_{\Lambda h_j,ae}(0)}.
\]
Writing $\mathcal E_{h_j}=z_j+A_jB_1$, the strip
\eqref{eq:fixed-strip} gives
\[
 |A_je|\le Ch_j,
 \qquad |e\cdot z_j|\le Ch_j.
\]
Hence, for large $j$,
\[
 |A_j^{-1}(x_j-z_j)|
 \ge\frac{|e\cdot(x_j-z_j)|}{|A_je|}
 \ge cR_j.
\]
By \eqref{eq:ellipsoid-critical-lower} and
\Cref{lem:appendix-convex-facts},
\[
 |S^v_{\Lambda h_j,ae}(0)|
 \ge cR_j h_j^{(n+s)/2}.
\]
This contradicts \eqref{eq:appendix-section-hypothesis} for the fixed
supporting slope $ae$, because $R_j\to\infty$.  Thus
\eqref{eq:axis-degeneration} holds.

\smallskip
\noindent\emph{Step 2: arbitrarily large lower $(n-s)$-density of $M_v$.}
We show that for every $x\in E$ and every $\varepsilon>0$ there are
arbitrarily small radii $r>0$ such that
\begin{equation}\label{eq:large-density}
  M_v(B_r(x))>\varepsilon^{-1}r^{n-s}.
\end{equation}
Fix a supporting slope at $x$, translate $x$ to the origin, and subtract its
supporting affine function.  Let $d_1(h)\ge\cdots\ge d_n(h)$ be the John
semiaxes of $S^v_{h,0}(0)$ and define
\[
 I:=\min\left\{i:\liminf_{h\downarrow0}
           \frac{d_i(h)}{\sqrt h}=0\right\}.
\]
Step 1 gives $I\le n-s$.  Given $\delta>0$, choose $h_j\downarrow0$ so that
\begin{equation}\label{eq:I-axis-choice}
 d_I(h_j)<\delta\sqrt{h_j},
 \qquad
 d_i(h_j)\ge\eta\sqrt{h_j}\quad(i<I),
\end{equation}
where $\eta>0$ is independent of $j$.

Use principal-axis coordinates for the John ellipsoid, let $V_j$ be the span
of the last $n-I+1$ directions, put $w_j=v|_{V_j}$, and set
\[
 D_j=S^{w_j}_{h_j,0}(0)
     =S^v_{h_j,0}(0)\cap V_j.
\]
Projection of the upper John ellipsoid onto $V_j$, together with
\eqref{eq:john-product-upper} and \eqref{eq:I-axis-choice}, gives
\begin{equation}\label{eq:slice-volume}
 |D_j|_{V_j}
 \le C\prod_{i=I}^n d_i(h_j)
 \le C\eta^{-(I-1)}h_j^{(n+s-I+1)/2}.
\end{equation}
The Alexandrov maximum estimate, \Cref{lem:appendix-alexandrov-maximum},
in the $(n-I+1)$-dimensional space $V_j$, applied to $w_j-h_j$ on $D_j$,
yields
\[
 M_{w_j}(D_j)|D_j|_{V_j}\ge c h_j^{n-I+1}.
\]
Therefore
\begin{equation}\label{eq:slice-mass}
 M_{w_j}(D_j)
 \ge c\eta^{I-1}h_j^{(n-s-I+1)/2}.
\end{equation}

The upper John inclusion and $0\in S^v_{h_j,0}(0)$ imply
$D_j\subset B_{C_nd_I(h_j)}(0)$.  Set
$r_j=4C_nd_I(h_j)$.  For large $j$, the ball $B_{r_j}$ is compactly
contained in $B_1$, and \Cref{lem:appendix-mass-lifting} with
$d=n-I+1$ gives
\begin{align}
 M_v(B_{r_j})
 &\ge c r_j^{I-1}h_j^{(n-s-I+1)/2}\notag\\
 &=c\left(\frac{\sqrt{h_j}}{r_j}\right)^{n-s-I+1}
      r_j^{n-s}\notag\\
 &\ge c\delta^{-(n-s-I+1)}r_j^{n-s}.
\label{eq:density-from-axis}
\end{align}
Since $n-s-I+1\ge1$ and $\delta>0$ is arbitrary,
\eqref{eq:density-from-axis} proves \eqref{eq:large-density}.

\smallskip
\noindent\emph{Step 3: covering.}
Fix $\beta>0$ and $\varepsilon>0$.  By \eqref{eq:large-density}, every point
of $E\cap B_{1-\beta}$ is the center of arbitrarily small balls
$B_r(x)\subset B_{1-\beta/2}$ satisfying
\[
 r^{n-s}<\varepsilon M_v(B_r(x)).
\]
Vitali's theorem gives a pairwise disjoint countable family
$\{B_{r_i}(x_i)\}$ whose fivefold enlargements cover
$E\cap B_{1-\beta}$.  Since $v$ is locally Lipschitz,
$M_v(B_{1-\beta/2})<\infty$, and hence
\[
 \sum_i(10r_i)^{n-s}
 \le C_n\varepsilon\sum_iM_v(B_{r_i}(x_i))
 \le C_n\varepsilon M_v(B_{1-\beta/2}).
\]
The radii can be required to be smaller than any prescribed number.  Letting
$\varepsilon\downarrow0$ gives
\[
 \cH^{n-s}(E\cap B_{1-\beta})=0.
\]
Finally let $\beta\downarrow0$ through a countable sequence.
\end{proof}

\begin{remark}
When $s=1$, the argument is Mooney's proof.  The additional observation for
$s>1$ is the lower bound $d_i(h)\gtrsim h$ for every John axis, obtained from
local Lipschitz continuity.  If $d_{n-s}(h)\gtrsim\sqrt h$, this forces the
John ellipsoid to have the critical volume $\gtrsim h^{(n+s)/2}$; tilting to
an extremal supporting slope then enlarges that volume by an unbounded
factor.
\end{remark}

\clearpage
\section{A local form of the Chou--Wang Pogorelov estimate}
\label{app:localized-chou-wang}

The regularity argument uses the bounded-domain form of the Pogorelov estimate
recorded by Mooney in \cite[Theorem~2.5 and Remark~2.6]{MooneyRemarks2025}.
Chou and Wang state their estimate with a $k$-admissible comparison function
\cite[Theorem~4.1]{ChouWang2001}.  The two points that are not written out in
Mooney's note are that the calculation localizes to arbitrary bounded domains
and that the comparison-function hypothesis may be weakened to a linearized
inequality.  We verify both points here.  The long maximum-principle
calculation itself remains the published calculation of Chou and Wang.

For a smooth $k$-admissible function $u$, write
\[
  L_u\phi:=\sigma_k^{ij}(D^2u)\phi_{ij},
  \qquad
  F(A):=\sigma_k(A)^{1/k}.
\]

\begin{lemma}[The analytic input from Chou--Wang]
\label{lem:CW-interior-core}
Let $G\subset\R^n$ be a bounded smooth domain.  Suppose that
$u,\rho\in C^\infty(\overline G)$, that $D^2u\in\Gamma_k$, and that
\[
  \sigma_k(D^2u)=1\quad\text{in }G,
  \qquad
  \rho>0\quad\text{in }G,
  \qquad
  \rho=0\quad\text{on }\partial G.
\]
Assume further that
\begin{equation}\label{eq:CW-core-input}
  F^{ij}(D^2u)\rho_{ij}\ge-1\quad\text{in }G.
\end{equation}
Then
\begin{equation}\label{eq:CW-interior-core}
  \sup_G\rho^4|D^2u|
  \le C\!\left(n,k,
    \|Du\|_{L^\infty(G)},
    \|D\rho\|_{L^\infty(G)}\right).
\end{equation}
No boundary curvature enters this estimate.
\end{lemma}

\begin{proof}
Chou and Wang maximize
\[
  \rho^4\,
  \phi\!\left(\frac{|Du|^2}{2}\right)u_{\xi\xi}
\]
and derive their equations (4.4)--(4.14) in the proof of
\cite[Theorem~4.1]{ChouWang2001}.  The factor $\rho^4$ makes every positive
maximum interior.  Equations (4.4)--(4.8) use only $\rho$ and $D\rho$, the
differentiated equation, concavity of $F$, and the symmetric-function
inequalities established earlier in that paper.  The only occurrence of the
second derivatives of the comparison function is equation (4.9), whose
precise input is
\[
  F^{ij}\rho_{ij}\ge-F^{ij}u_{ij}=-1.
\]
After this inequality, equations (4.10)--(4.14) use no further property of
the comparison function.  Replacing their equation (4.9) by
\eqref{eq:CW-core-input} therefore gives \eqref{eq:CW-interior-core} with the
same calculation.  Since the maximum is interior, no second fundamental form
or other geometric datum of $\partial G$ is used.
\end{proof}

\begin{proposition}[Localized Chou--Wang estimate]
\label{prop:localized-chou-wang}
Let $D\subset\R^n$ be bounded and open.  Let
\[
 u,w\in C^\infty(D)\cap C(\overline D),
 \qquad D^2u\in\Gamma_k,
 \qquad \sigma_k(D^2u)=1\quad\text{in }D.
\]
Assume that
\[
  w>u\quad\text{in }D,
  \qquad w=u\quad\text{on }\partial D,
  \qquad L_uw\ge0\quad\text{in }D,
\]
and that $Du,Dw\in L^\infty(D)$.  Then
\begin{equation}\label{eq:localized-CW}
  \sup_D(w-u)^4|D^2u|
  \le C\!\left(n,k,
    \|Du\|_{L^\infty(D)},
    \|Dw\|_{L^\infty(D)}\right).
\end{equation}
In particular, no smoothness, convexity, or curvature assumption on
$\partial D$ is required.
\end{proposition}

\begin{proof}
Put $\rho=w-u$.  Along the solution $\sigma_k(D^2u)=1$,
\begin{equation}\label{eq:Fij-scaled}
  F^{ij}(D^2u)=\frac1k\sigma_k^{ij}(D^2u).
\end{equation}
Thus $L_uw\ge0$ gives $F^{ij}w_{ij}\ge0$.  Euler's identity for the
one-homogeneous function $F$ gives
\[
  F^{ij}(D^2u)u_{ij}=F(D^2u)=1,
\]
and hence
\begin{equation}\label{eq:key-CW-input}
  F^{ij}(D^2u)\rho_{ij}\ge-1.
\end{equation}
Equivalently,
\[
  L_u\rho\ge-k,
  \qquad L_uu=k.
\]
This proves the weakening observed in
\cite[Remark~2.6]{MooneyRemarks2025}: in the Chou--Wang proof,
admissibility of $w$ is used only to obtain the nonnegative contraction
$F^{ij}w_{ij}$.

It remains to localize.  Let $\varepsilon>0$ be a regular value of $\rho$ and
set
\[
  D_\varepsilon:=\{x\in D:\rho(x)>\varepsilon\}.
\]
Because $\rho\in C(\overline D)$ and $\rho=0$ on $\partial D$,
\[
  \overline{D_\varepsilon}\Subset D.
\]
Every connected component $G$ of $D_\varepsilon$ is a bounded smooth domain.
On $G$, the function $\rho_\varepsilon:=\rho-\varepsilon$ is positive,
vanishes on $\partial G$, has the same first and second derivatives as
$\rho$, and satisfies \eqref{eq:key-CW-input}.  The restrictions of $u$ and
$\rho_\varepsilon$ are smooth on a neighborhood of $\overline G$.
\Cref{lem:CW-interior-core} therefore applies on every component with the
same constant, giving
\[
  \sup_{D_\varepsilon}(\rho-\varepsilon)^4|D^2u|
  \le C\!\left(n,k,
    \|Du\|_{L^\infty(D)},
    \|Dw\|_{L^\infty(D)}\right).
\]
By Sard's theorem, choose positive regular values
$\varepsilon_j\downarrow0$.  For each fixed $x\in D$, one has
$x\in D_{\varepsilon_j}$ for all large $j$.  Letting $j\to\infty$ and then
taking the supremum over $x$ proves \eqref{eq:localized-CW}.
\end{proof}

\begin{corollary}[Admissible-barrier form]
\label{cor:admissible-CW-barrier}
In \Cref{prop:localized-chou-wang}, the hypothesis $L_uw\ge0$ is automatic
if $D^2w\in\Gamma_k$.
\end{corollary}

\begin{proof}
The function $F=\sigma_k^{1/k}$ is concave and one-homogeneous on $\Gamma_k$.
For $A,B\in\Gamma_k$,
\[
  F(B)\le F(A)+DF(A):(B-A)=DF(A):B.
\]
Taking $A=D^2u$ and $B=D^2w$ yields
$F^{ij}(D^2u)w_{ij}\ge F(D^2w)>0$.  By
\eqref{eq:Fij-scaled}, this is equivalent to $L_uw>0$.
\end{proof}

\begin{remark}[Use in the main proof]
\label{rem:CW-current-application}
On the regular level-set domains $D_j$ in \Cref{lem:barrier-regularity}, the
comparison function is the quadratic barrier $w-t_j$.  Its Hessian belongs
to $\Gamma_k$, so \Cref{cor:admissible-CW-barrier} applies directly.  Thus
the regularity proof uses the admissible-barrier form of the published
Chou--Wang estimate; the linearized formulation above is included to justify
the stronger localized theorem recorded by Mooney.
\end{remark}

\clearpage
\bibliographystyle{amsalpha}
\bibliography{references}

\end{document}